\documentclass[11pt]{article}

\title{The Liouville property for groups acting on rooted trees}

\author{Gideon Amir\thanks{Bar-Ilan University,  gidi.amir@gmail.com}
  \and Omer Angel\thanks{University of British Columbia, angel@math.ubc.ca}
  \and Nicol\'as Matte Bon\thanks{Universit\'e Paris Sud,
    nicolas.matte.bon@ens.fr}
  \and B\'alint Vir\'ag\thanks{University of Toronto, balint@math.utoronto.ca}}

\date{October 2014}

\usepackage{amssymb,amsfonts,amsthm,amsmath}
\usepackage{graphicx}
\usepackage{wrapfig}
\usepackage{float}
\usepackage{geometry}
\usepackage{pst-all}
\definecolor{darkblue}{rgb}{0,0,0.44} 
\usepackage[pdfborder={0 0 0},pdfborderstyle={},colorlinks=true,linkcolor=darkblue,citecolor=black,urlcolor=blue]{hyperref}

\usepackage[font=sf, labelfont={sf,bf}, margin=1cm]{caption}

\renewcommand{\S}{{S}}
\newcommand{\R}{\mathbb{R}}
\newcommand{\G}{\mathcal{G}}
\newcommand{\M}{{M}}
\renewcommand{\H}{{H}}
\newcommand{\aut}{\operatorname{Aut}}
\newcommand{\supp}{\operatorname{supp}}

\newcommand{\T}{\mathbb{T}}
\newcommand{\mb}{{\overline{m}}}
\newcommand{\Tmb}{\T_\mb}

\newcommand{\al}{\mathbf}
\newcommand{\A}{{A}}
\newcommand{\Ag}{\mathcal{A}}
\newcommand{\B}{{B}}
\newcommand{\Abb}{\mathbb{A}}
\renewcommand{\Bbb}{\mathbb{B}}
\renewcommand{\O}{\mathbb{O}}
\newcommand{\E}{\mathbb{E}}
\newcommand{\N}{\mathbb{N}}
\newcommand{\bg}{\mathbf{g}}
\newcommand{\bv}{\mathbf{v}}
\newcommand{\W}{\mathbb{W}}
\renewcommand{\P}{\mathbb{P}}
\newcommand{\res}{\operatorname{Res}}
\newcommand{\Aut}{\operatorname{Aut}}
\newcommand{\Ubb}{\mathbb{U}}

\usepackage[nameinlink]{cleveref}   
  \crefname{theorem}{Theorem}{Theorems}
  \crefname{thm}{Theorem}{Theorems}
  \crefname{main thm}{Theorem}{Theorems}
  \crefname{lemma}{Lemma}{Lemmas}
  \crefname{claim}{Claim}{Claims}
  \crefname{remark}{Remark}{Remarks}
  \crefname{proposition}{Proposition}{Propositions}
  \crefname{defin}{Definition}{Definitions}
  \crefname{definition}{Definition}{Definitions}
  \crefname{corollary}{Corollary}{Corollaries}
  \crefname{section}{Section}{Sections}
  \crefname{figure}{Figure}{Figures}
  \newcommand{\crefrangeconjunction}{--}

\theoremstyle{definition}
  \newtheorem{defin}{Definition}[section]

\theoremstyle{plain}
  \newtheorem{thm}[defin]{Theorem}
  \newtheorem{main thm}{Theorem}
  \newtheorem{prop}[defin]{Proposition}
  \newtheorem{coro}[defin]{Corollary}
  \newtheorem{lemma}[defin]{Lemma}
  \newtheorem{claim}[defin]{Claim}

\theoremstyle{remark}
  \newtheorem{remark}[defin]{Remark}

\begin{document}

\maketitle

\begin{abstract}
  We show that on groups generated by bounded activity automata, every
  symmetric, finitely supported probability measure has the Liouville
  property.  More generally we show this for every group of automorphisms
  of bounded type of a rooted tree.  For automaton groups, we also give a
  uniform upper bound for the entropy of convolutions of every symmetric,
  finitely supported measure.
\end{abstract}

\section{Introduction}

Groups acting on rooted trees are a source of finitely generated groups
with a number of interesting properties concerning amenability, growth and
random walks.  An important special case are automata groups.  These
include the Grigorchuk group of intermediate growth \cite{Grigorchuk83},
Gupta and Sidki's examples of finitely generated torsion $p$-groups
\cite{Gupta-Sidki}, the Hanoi Tower groups \cite{Hanoi}, the Basilica group
\cite{Basilica} and iterated monodromy groups arising from holomorphic
dynamics: see \cite{Nek05, Nek11} for a survey of the topic and literature.

One aspect of groups acting on rooted trees that has attracted some attention is the
behaviour of random walks on them, especially random walks with the
\emph{Liouville property}. (For the definition and preliminaries on the
Liouville property see \cref{S:Liouville}.) On one hand, construction based on these groups have provided new examples of asymptotic behaviours for the rate of
escape and entropy of random walks in the sublinear range \cite{Brieu:entropy, AV:speedexponents}
and for the relationship between the Liouville property and growth of
groups \cite{Erschler:boundarybehavior, BE:poisson}.  On the other hand the
Liouville property has turned out to be a useful tool to prove
\emph{amenability} of several classes of groups acting on rooted trees
\cite{BV05, Kai:munchaussen, Brieu:nonuniformgrowth, BKN10,
  AAV:amenability}.  The motivation of this paper has its roots in these
results, that we now outline.

The first result of this kind is due to Bartholdi and Vir\'ag \cite{BV05},
who used random walks to prove amenability of the \emph{Basilica group}.  A
key observation was that the Basilica group admits a special, so-called
\emph{self-similar}, non-degenerate symmetric, finitely supported measure.
From this, they deduce amenability of the group.  Their method was later
generalized and simplified by Kaimanovich \cite{Kai:munchaussen} who gave a
general definition of a self-similar measure on a group acting on a rooted
tree, and showed that any such measure has the Liouville property.  In
particular a group supporting a non-degenerate self-similar measure is
amenable.  As noted in \cite{Kai:munchaussen}, existence of a finitely supported
self-similar measure is rare as it relies on strong combinatorial
assumptions on the action on the tree.

These ideas were further developed and used to show amenability of a large
class of automata groups, namely groups generated by \emph{finite automata
  of bounded activity} by Bartholdi, Kaimanovich and Nekrashevych \cite{BKN10}
and groups generated by \emph{finite automata of linear activity} in
\cite{AAV:amenability} (see \cref{S:automata} for definitions regarding
automata groups and their activity degree).  The idea of the proofs in
\cite{BKN10, AAV:amenability} is to embed all such groups in a special
family of groups called the \emph{mother groups}, and show that these admit
a special generating measure with the Liouville property.  Thus the mother groups
are amenable.  Since amenability is inherited by subgroups, so are all
groups generated by finite automata of bounded and linear activity.  The measure considered on the
mother groups also has a certain self-similarity, weaker than
the self-similarity required in \cite{Kai:munchaussen}.  These results have
been unified and are now part of a more general amenability criterion due
to Juschenko, Nekrashevych and de la Salle \cite{JNS:recurrentgrupoids},
which applies to a wider class of groups (and is not concerned with the
Liouville property).

The results and methods from \cite{BV05,Kai:munchaussen, BKN10,
  AAV:amenability} raise the following question: does the Liouville
property hold for \emph{all} symmetric, finitely supported measures on the
groups considered there?  Note that the results above do not even imply
that every bounded or linear automaton group \emph{admits} a symmetric,
finitely supported, generating probability measure with the Liouville property (see
\cref{S:Liouville} for an account on open questions concerning
the stability of the Liouville property).  A positive answer was conjectured
in \cite{AAV:amenability} for groups generated by finite automata of
bounded, linear and quadratic activity.  It is shown in
\cite{AV:positivespeed} that this does not hold in general for automata
groups of polynomial activity of degree at least 3, and it is not yet known
if these groups are amenable as asked by Sidki \cite{Si04}.

\subsection{Statement of results}

The aim of this paper is to to establish the Liouville
property for a large class of random walks on groups acting on rooted trees,
that may lack self-similar properties.  Our method combines ideas from
papers cited above, together with an analysis of the orbital Schreier
graphs for the group action on the rooted tree.  Our first result gives a
partial answer to the conjecture in \cite{AAV:amenability}, covering the
case of bounded automata groups.  In fact we do not need the assumption that
the groups are generated by a \emph{finite state} automaton: our result
applies to general groups of automorphisms of \emph{bounded type} (see
\cref{D: bounded automorphisms}) of a spherically homogeneous rooted tree.

\begin{main thm}\label{T:bounded Liouville}
  Let $G$ be a group of automorphisms of bounded type of a spherically
  homogeneous rooted tree of bounded valencies.  Then every symmetric,
  finitely supported probability measure $\mu$ on $G$ has the Liouville
  property.
\end{main thm}

Amenability of general groups of automorphisms of bounded type of a rooted
tree is a particular case of the result of Juschenko, Nekrashevych and de
la Salle \cite{JNS:recurrentgrupoids} which answers a question of
Nekrashevych \cite{Nekrashevych:freesubgroups}.  Since the Liouville property
implies amenability, \cref{T:bounded Liouville} also implies this result.

A key ingredient of the proof of \cref{T:bounded Liouville} is
\emph{recurrence} of the orbital Schreier graphs for the action on the
boundary of the rooted tree \cite{Bo07,JNS:recurrentgrupoids}.  More can be
said in cases where the Schreier graphs have explicit descriptions.  In
such cases a closer analysis of the Schreier graphs yields explicit upper
bounds for the entropy of the convolutions $H(\mu^{*k})$ (see
\cref{sec:entropy} for preliminaries regarding entropy and its relationship
to the Liouville property).

We illustrate this with the \emph{principal group of directed automorphism}
$M(A,B)$ (see \cref{mother group} for the definition).  These groups were first
defined and studied by Brieussel, who proved their amenability in
\cite{Brieu:nonuniformgrowth} using random walks and in
\cite{Brieu:folnersets} by exhibiting F{\o}lner sets.  These are
generalizations of the mother groups from \cite{BKN10}.  In particular they
contain as subgroups all groups generated by finite-state automata with
bounded activity (see \cref{embedding mother group} below).  A particular
case of the group $\M(A,B)$ was also used in \cite{Brieu:entropy,
  AV:speedexponents}.

\begin{main thm}\label{T:spherically homogeneous}
  Let $\M(\A,\B)$ be a group as in \cref{def:group of directed
    automorphisms} acting on a spherically homogeneous rooted tree $\Tmb$
  with bounded valencies $\bar{m}=(m_n)$.  Then every symmetric, finitely
  supported measure $\mu$ on $\M(\A,\B)$ has the Liouville property.
  Moreover there exists a constant $C$ depending on $\supp(\mu)$ only such
  that
  \[
  H(\mu^{*k}) \leq C k^\alpha.
  \]
  where $\alpha=\log m_*/\log\frac{m_*^2}{m_*-1}$ and $m_*=\max(\bar{m})$.
\end{main thm}

In \cite{BKN10} this bound was obtained in the case when $M(A, B)$ is the mother
group and $\mu$ is in a special class of measures defined there.  Note the
support of $\mu$ need not generate all of $M(A,B)$.  Since every group
generated by a finite automaton of bounded activity is a subgroup of some
group of the form $M(A,B)$, we  get the following corollary.

\begin{coro}
  Let $G$ be a group generated by a finite automaton of bounded activity,
  and $\mu$ a symmetric, finitely supported probability measure on $G$. Then
  $H(\mu^{*k}) \leq C k^\alpha$, where $\alpha<1$ depends only on the group
  $G$ and the constant $C$ depends only on the support of $\mu$.
\end{coro}

The exponent $\alpha$ can be explicitly determined from the structure of
the automaton by following the argument in \cite[Theorem 3.3]{BKN10} to embed $G$ in the mother group.
\bigskip

A comment on the linear and quadratic activity case in  the conjecture in \cite{AAV:amenability} seems in order. Can these cases be attacked using the method of this paper?
The ascension diagrams (see \cref{multi-ascension}) become more complicated. To analyse them effectively, a more precise understanding of simple random walk on the  Schreier graphs seems needed, beyond the fact that the infinite graphs are recurrent.  This task becomes harder together with the level of precision required, as the graphs also become more complicated. We believe that this can be done to prove the conjecture in the linear case. However this would require a considerably  more complicated analysis relying on quantitative resistance estimates. We do not know if there is any hope to apply our method to the quadratic case.

\subsection{Preliminaries on the Liouville property}
\label{S:Liouville}

Given a probability measure $\mu$ on a countable group $G$, a function
$f:G\rightarrow \R$ is said to be $\mu$-harmonic if $f(g)=\sum_{h\in
  G}f(gh)\mu(h)$ for every $g\in G$. The measure $\mu$ is said to have the
\emph{Liouville property} if every bounded $\mu$-harmonic function on $G$
is constant on the subgroup $H = \langle \operatorname{supp}(\mu) \rangle$.
An equivalent formulation of the Liouville property is triviality of the
\emph{Poisson boundary} of $(H,\mu)$ \cite{KV83}.  If moreover the measure
$\mu$ is symmetric and has finite first moment with respect to a word
metric, the Liouville property is equivalent to the random walk with step
measure $\mu$ having $0$ asymptotic speed \cite[Corollary 3]{KL07}.  Under
the weaker assumption that $\mu$ has finite entropy, the Liouville property
is equivalent to vanishing of the asymptotic entropy $h(\mu)$ (\cite{KV83,
  Dierrenic}).  The latter will be the characterisation of the Liouville
property that we use in this paper.  See \cref{sec:entropy} for
preliminaries regarding entropy.

Amenability of a countable group $G$ is equivalent to the existence of a
Liouville symmetric measure supported on a generating set of $G$
\cite{KV83, Rosenblatt}.  This measure may have infinite support, as in
the well-known case of the Lamplighter group over $\mathbb{Z}^3$, namely
$\mathbb{Z}/2\mathbb{Z}\wr\mathbb{Z}^3$, see\cite{KV83}.  In some amenable
groups it must even have infinite entropy, see \cite{Erschler:liouville}.
Thus existence of a \emph{finitely supported} Liouville symmetric measure
whose support generates the group is strictly stronger than amenability. The Liouville property depends on the choice of $\mu$; however it
is an important open question whether it is a group
property when one restricts to symmetric measures with finite
generating support, and whether it is inherited by subgroups for the same class of measures.

\subsection{Structure of the paper and overview of the proofs}

\cref{sec:trees} contains preliminaries on groups acting on rooted trees.

\cref{sec:RWIDF} contains general facts on random walks on groups acting on
rooted trees.  Most of this section is based on the connection between
groups acting on rooted trees and \emph{random walks with internal degrees
  of freedom}, introduced by Kaimanovich \cite{Kai:munchaussen}.  A random
walk with internal degrees of freedom on a group $G$ with space of degrees
$Y$ is a Markov chain on $G\times Y$ that can be described in terms of a
random walk on a \emph{diagram}: a finite graph with vertex set $Y$ where
edges are labelled by probability measures on $G$.  We revisit and
slightly generalise ideas in \cite{Kai:munchaussen} by considering the
\emph{ascension diagram} (see \cref{multi-ascension}), a smaller diagram
obtained by stopping the walker on the Schreier graph when it visits a
fixed subset of $Y$ (rather than a single vertex as in previous works).  We
then prove an inequality linking the asymptotic entropy of the random walk
on the group with the ascension diagrams.  In \cite{Kai:munchaussen,
  BKN10,AAV:amenability}, random walk with internal degrees of freedom
arising from self-similar random walks were used, via explicit calculations
using matrices with entries in the group algebra.  For random walks lacking
of self-similarity properties, these calculations become more complicated.
To avoid these we take advantage of recurrence of the Schreier graphs through a simple fact proven at the end of the section.

In \cref{sec:proof thm 1} we prove \cref{T:bounded Liouville}. The proof is
based on the tools introduced in \cref{sec:RWIDF}. A key observation is
that \emph{sections} of elements in the support of $\mu$ at high enough
levels of the tree belong either to a finite group of \emph{finitary
  automorphisms} or to a finite
\emph{groupoid of directed automorphisms} (a notion introduced in
\cref{sec:trees}).  Combined with recurrence of the orbital Schreier
graphs, this yields bounds on the asymptotic speed of random walks with internal degrees of freedom determined by
the ascension diagrams, which are used to bound the entropy of the
original random walk.

Finally, in \cref{sec:proof thm 2} we prove \cref{T:spherically
  homogeneous}.  The additional ingredient needed is a analysis of the
orbital Schreier graphs for the action on the finite level of the tree
using electric network theory.  We give lower bounds on effective
resistances between certain points in the graph, and use them to get
explicit entropy estimates through arguments similar to \cref{sec:proof thm
  1}.

\paragraph*{Acknowledgements.}
We thank Mikael de la Salle for pointing out an imprecision in the
statement of \cref{P:entropy} in a previous version.  GA's research was
supported by the Israel Science Foundation (grant No. 1471) and by a Grant
from the GIF, the German-Israeli Foundation for Scientific Research and
Development.  OA was partially supported by NSERC and ENS in Paris. NMB was
introduced to this subject by Anna Erschler, and thanks her for several
conversations. The work of NMB was partially supported by the ERC staring grant GA 257110 ``RaWG''.

\section{Rooted trees and their automorphisms}
\label{sec:trees}

\subsection{Spherically homogeneous rooted trees and their automorphisms}

Let $\mb = (m_i)_{i\geq 1}$ be a bounded sequence of positive integers.
The \emph{spherically homogeneous rooted tree} $\Tmb$ is the tree where
each vertex at level $k$ has $m_{k+1}$ children in level $k+1$.  The tree
$\T_\mb$ has a root in level $0$, which is denoted $\varnothing$.  A vertex
at level $k$ is naturally encoded by a word $x_k x_{k-1}\dots x_1$, where
$x_i \in X_{m_i} = \{0,\dots,m_i-1\}$. The children of $v$ are words of the
form $xv$ where $x$ is a single letter.  We denote by $\Tmb^n \subset \Tmb$
the set of words of length $n$, i.e.\ the $n$th level of the tree.  Note
that words are read from right to left.

We denote by $\Aut(\Tmb)$ the group of automorphisms of $\Tmb$ that fix the
root.  Note that for some sequences $\mb$ (in particular the constant
sequences) all automorphisms of $\Tmb$ fix the root.  However, there are
sequences for which the tree has additional automorphisms which do not fix
the root and so do not belong to $\Aut(\Tmb)$ in our notations.  We write
actions of automorphisms on the right and use the notation
\[
w \mapsto w\cdot g
\]
for $w\in\Tmb$ and $g\in\Aut(\Tmb)$.
For $w\in\Tmb$, consider the sub-tree rooted at $w$.  If $w$ is at level
$n$, then this sub-tree is isomorphic to the spherically homogeneous rooted
tree $\T_{\sigma^n\mb}$, where $\sigma$ denotes the shift operator
\[
\sigma(m_1,m_2,\dots) = (m_2,m_3,\dots).
\]
Automorphisms $g\in\Aut(\Tmb)$ preserve the levels of the tree, so that
every word $w\in\Tmb$ is mapped by $g$ to a word $w\cdot g$ of the same
length, say $n$.  Since the sub-trees above $w$ and $w\cdot g$ are
canonically isomorphic, $g$ induces a bijection of the sub-trees rooted at
$w$ and $w\cdot g$, which can be identified with a unique element of
$\Aut(\T_{\sigma^n\mb})$.  This element is called the \emph{section} of $g$
at $w$ and it is denoted $g|_w$.  Formally, the section is the unique
element $g|_w \in \Aut(\T_{\sigma^n\mb})$ such that for every word
$v\in\T_{\mb}^n$,
\[
vw\cdot g = (v\cdot g|_w)(w\cdot g),
\]
where the parenthesis juxtaposition denotes concatenation of words.  It
immediately follows from the definition that sections are multiplied and
inverted according to the following rules
\begin{equation}\label{sections multiplication rule}
  (gh)|_w = g|_w h|_{w\cdot g};
  \qquad
  g^{-1}|_w=(g|_{w\cdot g^{-1}})^{-1}.
\end{equation}
Using an equivalent terminology, there is an isomorphism (a \emph{wreath
  recursion})
\[
\begin{array}{cccc}
  \Aut(\Tmb)&\to&\Aut(\T_{\sigma\mb}) \wr_{X_{m_1}}\S_{m_1}
    &= \Aut(\T_{\sigma\mb})^{X_{m_1}}\rtimes\S_{m_1},\\
  g&\mapsto& &(g|_0,\dots,g|_{m_1-1})\sigma.
\end{array}
\]
where $g|_0,\dots,g|_{m_1-1}$ are the first level sections of $g$ and the
permutation $\sigma$ gives its action on the first level $\Tmb^1 =
X_{m_1}$.

\begin{defin}\label{def:group of sections}
  Let $G < \aut(\Tmb)$.  For $n\in\N$ we denote by $G^{(n)}$ the group of
  {\em $n$th level sections} of $G$, i.e. the subgroup of
  $\Aut(\T_{\sigma^n\mb})$ generated by $\left\{ g|_w : g\in G,
    w\in\Tmb^n\right\}$.
\end{defin}

\begin{remark}\label{R: sections generators}
  If the group $G$ is generated by the set $S$, the groups of sections
  $G^{(n)}$ are generated by the $n$th level sections of elements in $S$,
  see \eqref{sections multiplication rule}.
\end{remark}

The action of $\Aut(\Tmb)$ naturally extends to an action by homeomorphism
on the \emph{boundary at infinity} of the tree $\partial \Tmb$.  The
boundary $\partial \Tmb$ is the set of infinite geodesic rays starting from
the root. In our notations it identifies with the set of left-infinite
sequences $\gamma=\cdots x_3x_2x_1$ where $x_i\in X_{m_i}$. The set
$\partial \Tmb$ is endowed with the natural product topology, which makes
it homeomorphic to the Cantor set.

\medskip

The {\em Schreier graph} associated with a group action is defined as
follows. If a group $G$ generated by a finite symmetric set $S$ acts on a
set $Y$, the Schreier graph has vertex set $Y$ and edges $(y,y\cdot s)$
for $y\in Y, s\in S$.  We admit that the action of $G$ on $Y$ can be non-transitive
and then the Schreier graph is disconnected.  A connected component of the
Schreier graph is called an {\em orbital} Schreier graph.

In our setting, a finitely
generated subgroup $G<\Aut(\T^d)$ naturally defines a sequence of finite
Schreier graphs arising from the action on the finite levels of the
tree.  It also defines a family of infinite graphs given by the orbital
Schreier graphs for the action of $G$ on $\partial \Tmb$.  The Schreier
graph for level $n+1$ covers the graph for level $n$.

\subsection{Activity and automorphisms of bounded type}
\label{S: activity}

The {\em activity function} of an automorphism $g\in\Aut(\Tmb)$ is the
function $\Gamma_g:\N\to\N$ that counts the number of level $n$ vertices
$v$ so that $g|_v\neq e$.  By \eqref{sections multiplication rule} the
activity satisfies
\[
\Gamma_{gh}(n)\leq \Gamma_g(n)+\Gamma_h(n); \qquad
\Gamma_{g^{-1}}(n)=\Gamma_g(n).
\]

This allows to define several subgroups of $\Aut(\Tmb)$ in terms of the
activity function.  For instance elements whose activity function is
bounded (respectively grows at most polynomially, respectively grows
subexponentially) form a subgroup of $\Aut(\Tmb)$.

\begin{defin}
  An element $g\in\Aut(\Tmb)$ is called \emph{finitary} if the sections
  $g|_v$ are non-trivial only for finitely many vertices $v\in \Tmb$. We
  define the \emph{depth} of $g$ to be the smallest level $n$ so that all
  sections at level $n$ are trivial.
\end{defin}

Finitary automorphisms of $\Tmb$ form a locally finite subgroup of
$\Aut(\Tmb)$.

Automorphisms of \emph{bounded type} are automorphims that
have bounded activity in a strong sense, that we now define.

\begin{defin}[Automorphism of bounded type]\label{D: bounded
    automorphisms}\label{D: singular rays}
  An automorphism $g\in\Aut(\Tmb)$ is said to be of \emph{bounded type} if
  there exists a finite set of rays in $\partial \Tmb$, called the
  \emph{singular rays} of $g$, and a $K>0$ so that $g|_v$ is finitary with
  depth at most $K$ whenever $v$ does not belong to a singular ray. The
  minimal such $K$ is called the \emph{depth} of $g$.
\end{defin}

In other word automorphisms of bounded type are those that have non-trivial
sections only in a bounded neighbourhood of a finite set of rays. Obviously
automorphisms of bounded type have bounded activity. In some special cases
the two notions coincide (for instance for automorphisms defined by a
finite-state automaton, see \cref{S:automata}).

\begin{remark}
  It is easy to see from \eqref{sections multiplication rule} that
  automorphisms of bounded type form a subgroup of
  $\operatorname{Aut}(\Tmb)$.
\end{remark}

\subsection{The groupoid of directed automorphisms of a rooted tree}
\label{S:groupoid}

A special case of automorphism of bounded type are the \emph{directed
  automorphisms}.

\begin{defin}
  An automorphism of bounded type $g\in \Aut(\Tmb)$ is said to be {\em
    directed} if it has at most one singular ray $\gamma\in\partial\Tmb$.
  If there is such a ray, we say $g$ is {\em directed along $\gamma$}.  By
  convention we say that a finitary $g$ is directed along every ray.
\end{defin}

Unlike automorphisms of bounded type, directed automorphisms do not form a
subgroup of $\Aut(\Tmb)$. However we have the following properties (cf.\
the section multiplication rule \eqref{sections multiplication rule}):
\begin{enumerate}
\item if $g,h $ are directed along $\gamma, \eta$ respectively and
  $\gamma\cdot g =\eta$ then $gh$ is directed along $\gamma$;
\item if $g$ is directed along $\gamma$ then $g^{-1}$ is directed along
  $\gamma\cdot g$.
\end{enumerate}

These properties suggest that the set of directed automorphisms of $\Tmb$
form essentially a \emph{groupoid} (up to some ambiguity originated by
finitary automorphisms).  We shall now make this intuition precise.

Recall that any right action of a group $G$ on a set $X$ defines a
groupoid, called the \emph{action groupoid} and denoted
$\mathcal{G}=\mathcal{G}(X, G)$. By definition $\mathcal{G}=X\times G$ as a
set; the product of two elements $(x, g)$ and $(y, h)$ in $\mathcal{G}$ is
defined whenever $x\cdot g=y$ and in this case $(x, g)(y, h)=(x, gh)$; the
inverse of $(x, g)$ is $(x\cdot g,\: g^{-1})$. Elements of the form $(x,
e)\in\mathcal{G}$ are called \emph{units}.  A \emph{subgroupoid} of
$\mathcal{G}$ is a subset $\mathcal{H}\subset \mathcal{G}$ which is closed
under taking inverses and products (i.e.\ whenever the product of two
elements in $\mathcal{H}$ is defined in $\mathcal{G}$, it belongs to
$\mathcal{H}$) and that contains all units $(y, e)$ for $y\in Y$, where $Y
\subset X$ is the projection of $\mathcal{H}$ to $X$ (note that this is
allowed to be a proper subset of $X$). The subgroupoid generated by a
family $\mathcal{S}\subset \mathcal{G}$ is the smallest subgroupoid
containing $\mathcal{S}$.

\newcommand{\ax}{\operatorname{o}}
\newcommand{\tg}{\operatorname{t}}
\newcommand{\Dmb}{\mathcal{D}_{\bar{m}}}

The \emph{groupoid of directed automorphism} of the rooted tree $\Tmb$ is
the subgroupoid $\Dmb$ of the action groupoid $\mathcal{G}(\partial
\Tmb, \Aut(\Tmb))$ which consists of couples $(\gamma, g)\in \partial
\Tmb\times \Aut (\Tmb)$ such that $g$ is directed along $\gamma$.
   We define the $\emph{depth}$ of $(\gamma, g)\in \Dmb$ to be the depth of $g$.
There is a natural projection
\begin{align*}
  \Dmb & \to \Aut(\Tmb)\\
  (\gamma, g)& \mapsto g,
\end{align*}
which maps the groupoid product, whenever defined, to the usual group product.

Any \emph{non-finitary} directed automorphism $g\in \Aut(\Tmb)$ has a
unique pre-image in $\Dmb$. This allows to think of $g$ either as an
element of $\Aut(\Tmb)$ or as an element of $\Dmb$. Finitary automorphisms
however have several pre-images and thus the groupoid $\Dmb$ cannot be
properly identified with a subset of $\Aut(\Tmb)$.

\begin{lemma}\label{L: locally finite groupoid}
  The groupoid $\Dmb$ is locally finite: every finite family
  $\mathcal{S}\subset \Dmb$ generates a finite subgroupoid. Moreover the
  cardinality of this finite subgroupoid has an upper bound which depends
  only on the cardinality of $\mathcal{S}$, on the maximal depth of
  elements in $\mathcal{S}$, and on $m_*$.
\end{lemma}

\begin{proof}
  It is a classical and elementary fact that the unrestricted infinite
  direct product of finite groups of bounded size is locally finite, and
  the cardinality of the subgroup generated by a finite subset has an upper
  bound that only depends on the size of the subset.

  Given a ray $\gamma=\cdots x_3x_2x_1\in \partial \Tmb$ define the element
  $h_\gamma\in\Aut(\Tmb)$ by the wreath recursion
  \[h_\gamma=(h_{\sigma\gamma}, e,\cdots, e)(0x_1),\] where
  $\sigma\gamma=\cdots x_3x_2$. This element is directed along the zero ray
  $\rho=\cdots 000$, moreover $\rho\cdot h_\gamma=\gamma$. It is straightforward to check
  that if $g$ is directed along $\gamma$ with depth at most $K$,
  then $h_\gamma g h^{-1}_{\gamma\cdot g}$ is directed along $\rho$, fixes
  $\rho$ and has depth at most $K$. The set of automorphisms of
  $\Aut(\Tmb)$ with these properties is isomorphic to an infinite direct
  product of finite groups with bounded cardinalities.

  It follows that if $\mathcal{S}\subset \Dmb$ is finite and consist of
  elements with depth at most $K$, elements of the form $(\rho,
  h_\gamma)(\gamma, s)(\gamma\cdot s, h_{\gamma\cdot s}^{-1})$ where
  $(\gamma, s)\in \mathcal{S}$ belong to an infinite direct product of
  finite groups of bounded cardinality. Let $\mathcal{H}\subset \Dmb$ be
  the subgroupoid that they generate, which is in fact a group. The cardinality of $\mathcal{H}$ has an
  upper bound that depends only on $|\mathcal{S}|$ and $K$. Now observe
  that the subgroupoid generated by $\mathcal{S}$ is contained in $\Cup_{\gamma, \eta}(\gamma,
  h_{\gamma}^{-1})\mathcal{H}(\rho, h^{-1}_{\eta})$ where $\gamma,\eta$ run
  along singular rays of elements in $\mathcal{S}\cup\mathcal{S}^{-1}$. The
  conclusion follows.
\end{proof}

From now on we adopt the following notation: calligraphic letters
(e.g. $\mathcal{S}$) will always denote subsets (or subgroupoids) of the
groupoid $\Dmb$, and we will sometimes denote with the corresponding
capital letter (e.g. $S$) the projection to the group $\Aut(\Tmb)$.

\subsection{Automata groups and their activity degree}
\label{S:automata}

We now recall some basic notions in the relevant particular case of
\emph{automata groups} acting on regular rooted trees. These notions do not play an
active role in the proofs, but this is the most relevant source of examples.

If the sequence $\mb$ is constant, equal to some positive integer $m$, the
tree $\T_m$ is called the \emph{regular rooted tree} of degree $m$. It is
indexed by the set of words in the alphabet $X=\{0,\dots,m-1\}$.

An important class of finitely generated groups acting on the tree $\T_m$
are groups generated by \emph{finite automata}.  An \emph{invertible
  automaton} over the alphabet $X$ is a set $A$ (the automaton state space)
together with a pair of maps
\begin{align*}
  A &\rightarrow \S_m   &  A\times X &\rightarrow A,\\
  a &\mapsto \sigma_a   &      (a,x) &\mapsto a_x.
\end{align*}
Such an automaton acts on words in the alphabet $X$ as follows: if the
current state is $a\in A$, and the automaton receives as input a letter $x$
it outputs the letter $x\cdot\sigma_a$, and switches to state $a_x$. Given
an initial state $a\in A$, any word input into the automaton yields an
output of equal length, and it is readily seen that for any initial state
this action defines an automorphism of $\T_m$.  This automorphism is as
follows: $a$ acts on the first level by the permutation $\sigma_a$; its
first level section at vertex $x$ is the automorphism defined by the state
$a_x\in A$.  If a state defines the identity automorphism $e$ of $\T_m$, it
is said to be \emph{trivial}.

Every automaton $A$ generates a subgroup of $\Aut(\T_m)$, generated by the
automorphisms corresponding to all states.  An equivalent description is
that we have a finite set $A\subset\Aut(\T_m)$, so that for any $g\in A$
and any $v\in \T_m$ we have $g|_v \in A$. Such a set naturally defines an
automaton.

We shall always suppose that automata are \emph{reduced}, i.e.\ two
distinct states of $A$ define distinct automorphisms of the tree.  Any
automaton can be brought to a reduced form by identifying states with the
same action on the tree.

If $a$ is a state of an automaton $A$, the \emph{activity function}
$\Gamma_a(n)$ (see \cref{S: activity}) is determined in a simple manner by
the structure of the automaton as we shall now explain.  First note that in
the automaton case $\Gamma_a(n)$ grows either polynomially with some
integer exponent $d_a$ or exponentially (in which case $d_a$ is set to be
$+\infty$).  (This is since these functions satisfy a linear recursion
among themselves, and since $A$ is finite.)  The \emph{activity degree} of
$A$ is defined to be $d=\max_{a\in A} d_a$.  This invariant was introduced
by Sidki in \cite{Si04}. When $d=0$ the automaton is said to be of
\emph{bounded activity}.  Some well-studied examples of groups acting on
rooted trees belong to the class of bounded activity automata groups,
including the Grigorchuk group, the Basilica group and iterated monodromy
groups of postcritically finite polynomials (see \cite{Nek05}).

An automaton gives rise to a directed graph, possibly with loops and
multiple edges, called the \emph{Moore diagram} of the automaton.  The
vertex set is $A$, and there is an oriented edge from $a$ to $a_x$ for
every $x\in X$.  This directed edge is labelled by $(x,\, x\cdot\sigma_a)$.
The trivial state is a sink.  For clarity, the loops based at the trivial
state are usually omitted from the Moore diagram.  See \cref{fig:hanoi} for
an example.


The activity degree can easily be computed by looking at the structure of
the Moore diagram.  A \emph{non-trivial simple cycle} (henceforth, just
cycle) in the diagram is a closed oriented path visiting each vertex at
most once which visits states other than the trivial state.  Note that a
path is its set of edges, so that it is possible for two distinct cycles to
visit the same vertices, and even in the same order.  The activity is
exponential ($d=+\infty$) if and only if some strongly connected component
of the Moore diagram contains more than one cycle (in particular
$d=+\infty$ whenever two distinct cycles intersect).  If this is not the
case, then there is a partial order on the set of cycles: say that $c\to
c'$ if there is an oriented path from some state in $c$ to some state in
$c'$.  The activity degree is then equal to the largest $d\geq 0$ for which
there are distinct cycles with $c_d \to c_{d-1} \dots \to c_0$.

It easily follows from this description that an automaton of bounded
activity generates a group of automorphisms of bounded type in the sense of
\cref{D: bounded automorphisms}.

\medskip

Another diagram associated to an automaton is the \emph{dual Moore
  Diagram}, which is a special case of a Schreier graph.  This is the
oriented graph that has the alphabet $X$ as vertex set, and for every $x\in
X$ and $a\in A$ there is an oriented edge going from $x$ to
$x\cdot\sigma_a$.  Such an edge is labelled by $(a,a_x)$.

The $n$th iteration of the dual Moore diagram is defined to be the oriented
graph that has as a vertex set the $n$th level of the tree $\T_m^n$ and for
every word $w\in\T_m^n$ and every state $a\in A$ there is an edge going
from $w$ to $w\cdot a$.  Such an edge is labelled by $(a,\, a|_w)$.  The
$n$th iteration of the dual Moore diagram is thus isomorphic as a graph to
the Schreier graphs of $G = \langle A\rangle$ acting on $\T_m^n$ with
generating set $A$.

\section{Tools for random walks on groups acting on rooted trees}
\label{sec:RWIDF}

\subsection{Random walk with internal degrees of freedom}
\label{S: general RWIDF}

Let $G$ be a group and $Y$ be a finite set.  Consider a Markov chain with
state space $Y$ and transition probabilities given by a stochastic matrix
\[
P=(p_{xy})_{x,y\in Y}
\]
We shall always suppose that this Markov chain is irreducible.  Consider
also a collection of probability measures on $G$, denoted $\mu_{xy}$ for
$x,y\in Y$.  These are called {\em edge measures}, and we denote the
collection by $M = (\mu_{xy})_{x,y\in Y}$.  Only measures $\mu_{xy}$ for
pairs with $p_{xy}\neq 0$ are used.

Note that our notation are different from those in \cite{Kai:munchaussen}
in that there $M=(\mu_{xy})$ denotes a matrix of \emph{sub}-probability
measures with total mass $p_{xy}$ and that equal our $\mu_{xy}$ only after
renormalization.

Given such a pair $(M,P)$ we draw the following \emph{diagram}: take the
(oriented) graph with vertex set $Y$ induced by stochastic matrix $P$ (with
an edge $(x,y)$ whenever $p_{xy}\neq 0$).  Label the edge $(x,y)$ by the
pair $(\mu_{xy},p_{xy})$.  We will hereinafter make no distinction between
$(M,P)$ and the associated diagram.

\begin{defin}
  The {\em random walk with internal degrees of freedom} corresponding to
  $(M,P)$ is the Markov chain $(\al{g}_k,\al{y}_k)$ on $G\times Y$, defined
  as follows: $\al{y}_k$ performs a random walk on $Y$ with transition
  probabilities given by $P$.  When $\al{y}_k$ crosses a given edge, the
  group element $\al{g}_k$ is multiplied on the right by a sample of the
  corresponding edge measure.  Formally, the transition probability from
  $(h,x)$ to $(g,y)$ is $p_{xy}\mu_{xy}(h^{-1}g)$.
\end{defin}


Recall that for a Markov chain with state space $S$ and $T\subset S$, the
induced Markov chain on $T$ has transition probabilities $p_{xy} =
\P_x(X_\tau=y)$ where $\tau = \inf\{n>0 : X_n\in T\}$.

\begin{defin}[Trace]\label{def:trace RWIDF}
  Let $(\al{g}_k,\al{y}_k)$ be a random walk with internal degrees of freedom
  on $G\times Y$ with diagram given by $(M,P)$.  For a non-empty  $W\subset
  Y$, the induced Markov chain on $G\times W$ is called the {\em trace over
    $W$} of the original random walk with internal degrees of freedom.
\end{defin}

It is easy to see that the trace of a random walk with internal degrees of
freedom is also a random walk with internal degrees of freedom $(M_W,P_W)$.
In general, the measures making up $M_W$ can be much more complex than the
measures in $M$. For example, they may have infinite support even if
measures of $M$ have finite support. However, the walks we study below are
such that we retain some control over the support of the new edge measures.

Note that the diagram $(M_W,P_W)$ of the trace does not depend on the
initial distribution of $(\al{g}_0,\al{y}_0)$.  Hence taking the trace
might be seen as an operation on diagrams. The diagram $(M_W,P_W)$ can be
explicitly computed from the diagram $(M,P)$ and formulae
can be given in term of matrices with entries in the group algebra
$\ell^1(G)$, as shown in \cite{Kai:munchaussen}.

\subsection{Entropy and speed of random walks with internal degrees of
  freedom}
\label{sec:entropy}

Let $\nu$ be a probability measure on a countable space $E$.  Recall that
its \emph{entropy} is the quantity
\[
H(\nu) = -\sum_{\nu(e)>0} \nu(e) \log \nu(e). 
\]
For a random variable $\al{X}$ taking values in a countable space, the
entropy $H(\al{X})$ is defined as the entropy of its distribution. Let us
recall some basic properties of entropy.

\begin{prop}\label{P:entropy}
  \begin{enumerate}
  \item If $\al{X}$ has finite support, then $H(\al{X}) \leq \log
    |\supp(\al{X})|$, and equality holds if and only if $\al{X}$ is
    uniformly distributed on $\supp(\al{X})$.
  \item Let $\al{Y},\al{X}_1,\dots,\al{X}_n$ be discrete random variables
    defined on the same probability space, and suppose that $\al{Y}$ is a
    function of $\al{X}_1,\dots,\al{X}_n$.  Then
    \[
    H(\al{Y}) \leq H(\al{X_1},\dots,\al{X}_n)
    \leq H(\al{X}_1) + \dots + H(\al{X}_n),
    \]
    where the middle term denotes the entropy of the joint distribution of
    $(\al{X}_1,\dots,\al{X}_n)$.
  \item Let $G$ be a group generated by a finite set $S$ with the shortest
    word metric $|\cdot|$.  There exists a constant $C$, depending only on
    $|S|$, such that if $\al{g}$ is a random variable taking values in $G$,
    then
    \[
    H(\al{g})\leq C\E |\al{g}|+C.
    \]
  \end{enumerate}
\end{prop}


Let $\mu$ be a probability measure on a group $G$, and $(\al{g}_k)_k$ be
the corresponding random walk.  By (2) above and sub-additivity, the
following limit exists:
\[
h(\mu) = \lim_{k\to\infty} \frac1k H(\mu^{*k})
= \lim_{k\to\infty} \frac1k H(\al{g}_k).
\]
The limit is called the \emph{asymptotic entropy} of $\mu$. The asymptotic
entropy is related to the Liouville property by the following fundamental
result:

\begin{thm}[\cite{KV83, Dierrenic}]
  Let $\mu$ have $H(\mu)<\infty$. Then $h(\mu)=0$ if and only if $(G,\mu)$
  has the Liouville property.
\end{thm}

Another fundamental quantity associated to $\mu$ is the \emph{asymptotic
  speed}. Let $G$ be generated by a finite $S$ and let $|\cdot|$ be the
associated word metric. The asymptotic speed with respect to $S$ is the
limit
\[
\ell_S(\mu)=\lim_{k\to\infty} \frac1k \E|\al{g}_k|,
\]
which exists by sub-additivity, provided $\mu$ has finite first moment
(i.e.\ $\sum |g| \mu(g) <\infty$).

The definitions of asymptotic entropy and speed extend to the setting of
random walks with internal degrees of freedom. Namely let
$(\al{g}_k,\al{y}_k)$ be random walk with internal degrees of freedom on
$G\times Y$ with diagram $(M,P)$.  Suppose that the initial distribution of
$(\al{g}_0, \al{y}_0)$ and all edge measures $\mu_{xy}$ have finite
entropy.  Then the asymptotic entropy of the random walk with internal
degrees of freedom is well defined and does not depend on the initial
distribution of $(\al{g}_0,\al{y}_0)$ (hence it is a numerical invariant of
the diagram):
\begin{equation}
  \label{eq:h_def}
  h(M,P) = \lim_{k\to\infty} \frac1k H(\al{g}_k,\al{y_k})
  = \lim_{k\to\infty} \frac1k H(\al{g}_k).
\end{equation}
Similarly the asymptotic speed $\ell_S(M,P)$ is well-defined whenever all
edge measures and the starting point $\al{g}_0$ have finite first moment.
Asymptotic speed and entropy are related by the inequality
\begin{equation}\label{E: fundamental inequality}
  h(M, P)\leq v_S\ell_S(M, P)\leq \log |S|\ell_S(M, P),
\end{equation}
where $v_S=\lim\frac1n\log(|S^n|)\leq \log|S|$ is the \emph{exponential
  growth rate} of the group $G$ with generating set $S$. We will only use
that the asymptotic entropy has a linear upper bound in terms of the speed,
with constant depending only on the number of generators.

If $(M,P)$ is a random walk with internal degrees of freedom on $G\times
Y$, let $\nu$ be the stationary distribution of $P$, which is unique since
we assume $P$ is irreducible.  If $(M_W,P_W)$ is the trace over
$W$, then the asymptotic entropies satisfy the relation (see
\cite[Proof of Theorem 3.3]{Kai:munchaussen}):
\begin{equation}\label{entropy of trace}
  h(M_W,P_W)= \frac{1}{\nu(W)}h(M,P).
\end{equation}
Note that the fraction of time spent in a subset $W$
converges a.s.\ to $\nu(W)$.

\subsection{Random walks with internal degrees of freedom and groups acting
  on rooted trees}
\label{RWIDF and groups acting on trees}

Let $\mu$ be a probability measure on a  group $G<\Aut(\Tmb)$ whose support
generates $G$, and consider the associated random walk  $(\al{g}_k)$.  Fix a
level $n\geq0$, and recall that $G^{(n)}$ denotes the subgroup of
$\Aut(\T_{\sigma^n\mb})$ generated by $n$th level sections of elements in $G$ (\cref{def:group of sections}).

Pick a vertex $v\in \Tmb^n$, and let $\O\subset\Tmb^n$ be its orbit under
the action of $G$.  Then $(v\cdot\al{g}_k)$ is a Markov chain
on $\O$, and  a key observation made in \cite{Kai:munchaussen} is that $(\al{g}_k|_v, v\cdot\al{g}_k)$ is a random walk with
internal degrees of freedom on $G^{(n)}\times\O$ (restricting to an orbit
assures the irreducibility condition for the marginal Markov chain). Let $(M, P)$ be its diagram.

It easily follows from the section multiplication rule \eqref{sections
  multiplication rule} that $(M,P)$ has transition probabilities and edge
measures given for every $v,w\in \O$ by
\begin{equation}
  \label{edge measures}
  \begin{split}
    p_{vw} &= \mu\{g : v\cdot g=w\}\\
    \mu_{vw}(h)&= \mu\{g : v\cdot g=w, g|_v=h\} /p_{vw}
    \quad\text{whenever }p_{vw}\neq 0.
  \end{split}
\end{equation}
Note moreover, that if $\mu$ is symmetric one has
$\mu_{vw}=\hat{\mu}_{wv}$, where $\hat{\nu}(g)=\nu(g^{-1})$ denotes the
reflected measure with respect to group inversion.  If $\mu$ is symmetric
and finitely supported, the diagram $(M,P)$ is isomorphic as a graph to the
Schreier graph of $G$ acting on $\O$ with generating set $\supp(\mu)$.
When $G$ is an automaton group, this diagram might also be seen as a
weighted version of the \emph{dual Moore diagram} of the $n$th iteration of
the automaton.

\begin{defin}[Ascension diagram]\label{multi-ascension}
  Let $G<\Aut(\Tmb)$, and $\mu$ be a probability measure on $G$ supported on
  a generating set.
  \begin{enumerate}
  \item Let $\O\subset \Tmb^n$ be a $G$-orbit.  We denote by
    $T_\O(\mu)=(M,P)$ the random walk with internal degrees of freedom on
    $G^{(n)}\times\O$, whose transition probabilities and edge measures are
    given by (\ref{edge measures}).
  \item More generally, let $\W\subset\O$ be non-empty.  We denote by
    $T_\W(\mu)=(M_\W,P_\W)$ the trace over $\W$ of $T_\O(\mu)$.
  \end{enumerate}
  We call $T_\W(\mu)$ the \emph{ascension diagram} of measure $\mu$ with
  respect to vertex set $\W$.  The case when $\W$ coincides with the whole
  orbit is seen as a particular case of the same definition.
\end{defin}

The simplest case of the above construction is when $\W=\{w\}$ is a
single point.  In this case $T_w(\mu)$ is just a new probability measure on
$G^{(n)}$, that admits a clear interpretation: it is the step measure of
the random walk on $G^{(n)}$ that one sees by looking to the action on the
subtree rooted at $w$ at the times when $w$ is stabilized (see
\cite{Kai:munchaussen,AAV:amenability}).  In this case $T_w$ is an operator
acting on measures and was called the \emph{ascension operator} in
\cite{AAV:amenability}.  The next theorem was stated and proved in
\cite{Kai:munchaussen}, in the above simpler situation and when the action
of $G$ on levels is transitive.

\begin{thm}\label{T:ascension inequality}
  Let $G<\Aut(\Tmb)$, and $\mu$ a measure on $G$ with finite entropy.  Let
  $\Tmb^n=\O_1\sqcup\cdots\sqcup \O_r$ be the partition of the $n$th level of
  the tree into $G$-orbits. Consider a collection of non-empty subsets
  $\W_i \subset \O_i$.  Then
  \[
  h(\mu)\leq \sum |\W_i| \cdot h(T_{\W_i}(\mu)).
  \]
\end{thm}

\begin{proof}
  Consider first the case that $\W_i=\O_i$ for every $i$.  The element
  $\al{g}_k$ is completely determined by its action on the
  $n$th level and its sections at vertices of that level, hence by the
  data of $(\al{g}_k|_v, v\cdot \al{g}_k)$ for every $v\in \Tmb^n$. By
  \cref{P:entropy}(2)
  \[
  H(\al{g_k})\leq \sum_{v\in X^n}
  H(\al{g_k}|_v\:,\:v\cdot\al{g}_k)
  = \sum_{v\in \O_1} H(\al{g}_k|_v\:,\:v\cdot\al{g}_k)+\cdots+\sum_{v\in \O_r}
  H(\al{g}_k|_v\:,\:v\cdot\al{g}_k).
  \]
  The latter are random walks with internal degrees of freedom with
  diagrams $T_{\O_i}(\mu)$. Dividing by $k$ and letting $k\to\infty$
  \[
  h(\mu)\leq\sum_{i=1}^r |\O_i| \cdot h(T_{\O_i}(\mu)).
  \]
  For general $\W_i\subset \O_i$ the theorem follows from relation
  \eqref{entropy of trace} and the observation that the stationary measure
  on each orbit is the uniform measure on it.
\end{proof}

\subsection{An illustrative example: the Hanoi Tower group}
\label{sec:hanoi}

Before turning to the proof of \cref{T:bounded Liouville} in full
generality, let us illustrate how the notions from the previous paragraph
are used in one particularly simple example --- the \emph{Hanoi Tower group}.
This group is generated by a 4-state automaton over the 3-elements
alphabet, and it is related to the classical Hanoi Tower game on 3 pegs.  Its
Schreier graphs on the levels of the tree are discrete approximation of the
Sierpinski gasket (see for instance \cite{Hanoi}).

The Hanoi group is the automaton group $G<\aut(\T_3)$ generated by the
three automorphisms of finite type $a,b,c$ defined by the wreath recursions
\begin{align*}
  a&=(a,e,e)(12)  &  b&=(e,b,e)(02)  &  c&=(e,e,c)(01).
\end{align*}
Note that $a^2=b^2=c^2=e$.  The Moore diagram of the automaton is shown in
\cref{fig:hanoi}.

One can prove that for every symmetric measure $\mu$ supported on any
generating set of $G$, and for every single vertex $w\in\T_3$, the
ascension operator $T_w(\mu)$ is infinitely supported.  In particular, $G$
admits no finitely supported \emph{self-similar measure} in the sense of
\cite{Kai:munchaussen}.  However the Liouville property can be shown as
follows.

Consider the uniform measure $\mu$ on the standard generators $\{a,b,c\}$.
The group $G$ acts transitively on the levels of the tree, so there is a
single orbit.  For every level $n$ set $\W_n=\{0^n,1^n,2^n\}\subset
\T_3^n$.  The diagram of $T_{\W_n}(\mu)$ is a triangle with self-loops.
The self-similarity of the generators $a,b,c$ (their sections are either
themselves or trivial) yields that $T_{\W_n}(\mu)$ has the same measures
$\mu_{xy}$ on the edges for every $n$. \cref{fig:hanoi} also shows the
ascension diagrams $T_{\O_n}(\mu)$ with respect to the whole orbit
$\O_n=\mathbb{T}_3^n$ and $T_{\W_n}(\mu)$ with respect to set $\W_n$.

\begin{figure}
  \begin{center}
    \vspace{-20pt}
    \includegraphics[width=0.74\textwidth]{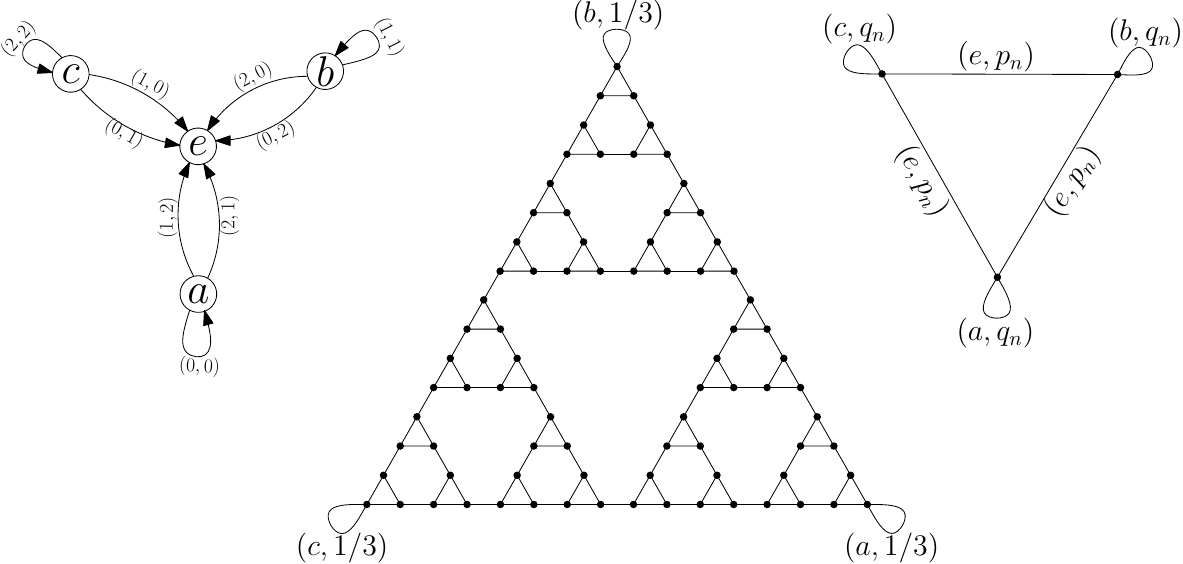}
  \end{center}
  \caption{Left: the Moore diagram for the Hanoi automaton.  Centre: The
    diagram for $T_{\O_n}(\mu)$ for $n=4$.  Here $a,b,c,e$ denote delta
    measures at those elements of the group.  Every edge except the
    self-loops at vertices $\W_n$ is labelled $(e,1/3)$.  Right: the
    diagram for $T_{\W_n}(\mu)$.}
  \label{fig:hanoi}
\end{figure}

The diagrams $T_{\W_n}(\mu)$ differ only in the transition probabilities
$p_n$ and $q_n$ (which satisfy $2p_n+q_n = 1$).  These can be determined in
turn by analysing simple random walk on the Schreier graphs of the group
$G$ acting on the levels of the tree (shown in the left).  It is easy to
see that these Schreier graphs converge to an infinite recurrent graph (in
the local topology, rooted at a vertex of $\W_n$). This implies that
$p_n\to 0$.  Since the generators are involutions, this roughly tells us
that the random walks with internal degrees of freedom $T_{\W_n}(\mu)$ get
``lazier'' as $n$ grows.  More precisely, using \eqref{E: fundamental inequality} one can
find a sequence of real numbers $\alpha_n$ decreasing to zero and prove an
\emph{a-priori} upper bound $h(T_{\W_n}(\mu)) \leq \alpha_n$ (we can have
$\alpha_n = C p_n$).  \cref{T:ascension inequality} then yields
\[
h(\mu) \leq 3 h(T_{\W_n}(\mu)) \leq 3 \alpha_n \to 0,
\]
which implies \emph{a-fortiori} that $h(\mu)=0$ (and also
$h(T_{\W_n}(\mu))=0$ for every $n$).

A similar argument actually applies to every symmetric and finitely
supported measure on the Hanoi group $G$, with a different choice of
$\W_n$.  We omit further details, as this is a special case of \cref{T:bounded Liouville}.

\subsection{Groups acting on a rooted tree with recurrent Schreier graphs}
\label{S: recurrent}

We say that a finitely generated subgroup $G$ of $\Aut(\Tmb)$ acts on
$\partial\Tmb$ \emph{with recurrent Schreier graphs} if every orbital
Schreier graph for the action of $G$ on $\partial\Tmb$ is recurrent.  Since
recurrence is stable under rough isometries, this property does not depend
on the choice of the finite symmetric generating set of $G$, and more
generally of a symmetric and finitely supported probability measure $\mu$
on $G$, see \cite[Theorem 2.17]{LyPe}.  Recurrence of the Schreier graphs
is related to groups of automorphisms of bounded type by the following result.

\begin{prop}[\cite{Bo07, JNS:recurrentgrupoids}]\label{T:bounded recurrent}
  Let $G$ be a finitely generated group of automorphisms of bounded type of
  a spherically homogeneous rooted tree $\Tmb$. Then $G$ acts on $\partial
  \Tmb$ with recurrent Schreier graphs.
\end{prop}

This fact was first proved by Bondarenko \cite{Bo07} for groups generated
by bounded automata; see \cite[Lemma 4.3]{JNS:recurrentgrupoids} for a
more general version which includes groups of
automorphisms of bounded type of a rooted tree.

\medskip

Let $G < \Aut(\Tmb)$ be a finitely generated group acting on $\partial
\Tmb$ with recurrent Schreier graphs, and endow it with a symmetric
finitely supported probability measure $\mu$.  Fix a starting ray $\gamma =
(v_0,v_1,\dots) \in \partial\Tmb$.

If $(\al{g}_k)_{k\in \N}$ is the random walk on $G$ driven by $\mu$, then
$(v_n \cdot \al{g}_k)$ is a Markov chain on level $n$ of the tree.  These
chains are naturally coupled, and the $n$th level Markov chain projects to
the previous ones. The Markov chain on the boundary of the tree
$(\gamma\cdot \al{g}_k)_{k\in \N}$ projects onto all of these.  Consider a
family of rays $\Ubb_\infty\subset \partial\Tmb$ containing $\gamma$ and
denote $\Ubb_n\subset \Tmb^n$ the set of projections of rays in
$\Ubb_\infty$ to the $n$th level of the tree. Let $\al{T}_\infty$ be the
first positive return time of $(\gamma \cdot \al{g}_k)$ to $\Ubb_\infty$
and let $\al{T}_n$ be the first positive return time of $(v_n\cdot \al{g}_k)$ to
$\Ubb_n$. Since the chains project onto each other, we have
\[
\al{T}_1\leq \al{T}_2\leq \al{T}_3\leq \dots \leq \al{T}_\infty.
\]
By recurrence of $(\gamma \cdot \al{g}_k)$, the sequence $(\al{T}_n)$ is bounded, and so
is constant for $n\geq R$ for some random $R$.  We therefore have
$\P(\al{T}_n \neq \al{T}_m) \leq \P(R > m\wedge n)$ which yields the
following proposition.

\begin{prop}\label{P: recurrence}
  With the notations above, the return times satisfy $\P(\al{T}_n \neq
  \al{T}_m)\to 0$ as $n,m\to\infty$ with $n, m\in \N\cup\{\infty\}$.
\end{prop}


\section{Proof of \cref{T:bounded Liouville}}
\label{sec:proof thm 1}

Throughout this section we fix a finitely generated group $G$ acting faithfully on
$\Tmb$ by automorphisms of bounded type, equipped with a finite, symmetric
generating set $S$. We also let $K>0$ be the maximal depth of a generator
$s\in S$ (see \cref{D: bounded automorphisms}).

\subsection{Deep level sections in groups of automorphisms of bounded type}
\label{S:sections_bounded}

The aim of this subsection is to construct generating sets for the groups of
level $n$ sections $G^{(n)}$ that have a special form adapted to our
purpose.  For every level $n$, denote by $\W_n\subset \Tmb^n$ the set of
vertices $w\in \Tmb^n$ such that the section $s|_w$ is non-trivial for some
$s\in S$.  Note that since each generator of $s$ has finitely many singular
rays, and non-trivial sections are all within distance $K$ of one of these rays,
the size of $\W_n$ is uniformly bounded in $n$.

Denote by $\Abb_\infty\subset \partial \Tmb$ the finite set of
rays of the tree which are singular for some generator $s\in S$ (see
\cref{D: singular rays}), and $\Abb_n$ its projection to level $n$ of the
tree.  The other non-trivial sections are at vertices of $\Bbb_n =
\W_n\setminus \Abb_n$.

\begin{remark}\label{R: gamma v}
  Observe that if $n$ is large enough, rays in $\Abb_\infty$ have distinct
  projections to level $n$. We assume henceforth that $n$ is large enough
  for this to hold.  For any such $n$ and every $v\in\Abb_n$ we denote by
  $\gamma_v\in \partial\T_{\sigma^n\mb} $ the continuation of this ray
  above $v$, i.e.\ the unique ray of the shifted tree $\T_{\sigma^n\mb}$
  such that $\gamma_v v \in \Abb_\infty$.
\end{remark}

For every $v\in\Abb_n$ and every $s\in S$ the section $s|_v$ is either
finitary or directed along $\gamma_v$.  Therefore $(\gamma_v,\:s|_v)$
belongs to the groupoid of directed automorphisms
$\mathcal{D}_{\sigma^n\bar{m}}$. Let $\Ag_n<\mathcal{D}_{\sigma^n\bar{m}}$
be the subgroupoid generated by $(\gamma_v,\:s|_v)$ when $v$ runs in
$\Abb_n$ and $s$ runs in $S$. By \cref{L: locally finite groupoid} the
groupoid $\Ag_n$ is finite, moreover its cardinality is uniformly bounded
in $n$ since the cardinality of $\Abb_n$ is bounded (in fact, it is
constant and equal to the cardinality of $\Abb_\infty$ if $n$ is large
enough) and all elements $(\gamma_v,\:s|_v)$ have depth at most $K$. Let
$\A_n$ be the projection of $\Ag_n$ to the group $\Aut(\T_{\sigma^n\mb})$,
i.e.
\[
A_n=\{g : \exists \gamma \text{ such that } (\gamma, g)\in \Ag_n \}.
\]

Consider now the case $v\in \Bbb_n$. Then for every $s\in S$ the section
$s|_v$ is finitary with depth at most $K$.  All such sections generate a
finite group; let us denote it $B_n$. This finite group also has uniformly
bounded cardinality.

Hence for every generator $s$ and every $v\in \W_n$ the section $s|_v$
belongs to $A_n\cup B_n$. It follows that the set $S_n=A_n\cup B_n$
generates the group of sections $G^{(n)}$ (see \cref{def:group of sections}
and \cref{R: sections generators}).

We summarize the discussion above as a proposition:

\begin{prop}\label{P:sections}
  For every large enough $n$, the group of sections $G^{(n)}$ admits a
  finite, symmetric generating set $S_n$ whose cardinality is bounded
  uniformly in $n$ and which can be written as a union $S_n=\A_n\cup\B_n$,
  where:
  \begin{itemize}
  \item $A_n$ is the projection to $\Aut(\T_{\sigma^n\mb})$ of a finite
    subgroupoid $\Ag_n$ of the groupoid $\mathcal{D}_{\sigma^n\bar{m}}$ of
    directed automorphisms of $\mathbb{T}_{\sigma^n\bar{m}}$;
  \item $B_n$ is a finite group of finitary automorphisms of $\T_{\sigma^n\mb}$.
  \end{itemize}
  Moreover, for every $s\in S$ we have that $s|_v\in A_n$ for $v\in\Abb_n$,
  $s|_v\in B_n$ for $v\in\Bbb_n$, and $s|_v=e$ otherwise.
\end{prop}

\subsection{Vanishing of asymptotic entropy}

We keep all notations introduced in the previous section: the vertex sets
$\W_n, \Abb_n, \Bbb_n$, the set of rays $\Abb_\infty$, and the generating
sets $S_n=A_n\cup B_n$ for the groups of sections.  We consider a
symmetric, finitely supported probability measure $\mu$ on $G$ with support
$S$.

Fix a level $n$ large enough so that \cref{R: gamma v} applies. Take any
orbit $\O\subset \Tmb^n$ for the action of $G$ and consider at first the
ascension diagram $T_\O(\mu)$ with respect to the orbit.  Let
$(\mu_{vw})_{v, w\in\O}$ be the edge measures of this ascension
diagram.  \cref{P:sections}, with the definition of the edge measures
\eqref{edge measures} and their symmetry property $\mu_{wv} =
\hat{\mu}_{vw}$ imply the following facts, which we summarize for later
reference.

\begin{claim}\label{C:mu_properties}
  \begin{enumerate}
  \item If $v,w\in \Abb_n$, then the measure $\mu_{vw}$ is supported in
    $\A_n$.  Moreover $(\gamma_v, h)$ belongs to the groupoid $\Ag_n$ for
    any $h\in \supp(\mu_{vw})$.
  \item If $v,w\in \Bbb_n$, then the measure $\mu_{vw}$ is supported in the
    finite group $\B_n$.
  \item Otherwise, $\mu_{vw}$ is concentrated on the identity.
  \end{enumerate}
\end{claim}

As a first consequence, observe that whenever the orbit $\O$ does not
intersect $\W_n=\Abb_n\sqcup\Bbb_n$ one has immediately $h(T_\O\mu)=0$,
since all edge measures of $T_\O$ are trivial and the corresponding random
walk with internal degrees of freedom is just a finite Markov chain.

Suppose now that there are $r=r(n)$ orbits in level $n$ that have non
trivial intersection with $\W_n$, and denote them $\O_{1,n}, \dots,
\O_{r,n}$.  Set $\W_{i,n}=\W_n\cap \O_{i,n}$, and consider the ascension
diagram $T_{\W_{i,n}}(\mu)$.  Note that the edge measure of this diagram
also satisfy \cref{C:mu_properties} (part 3 is vacuous here).
\cref{T:ascension inequality} and the above observation that
$h(T_\O(\mu))=0$ whenever $\O\cap\W_n = \varnothing$ give
\begin{equation}\label{main inequality}
  h(\mu) \leq \sum_{i=1}^{r} |\W_{i,n}| \cdot h(T_{\W_{i,n}}(\mu)).
\end{equation}
To prove that $h(\mu)=0$ we estimate the asymptotic speed of the diagrams
$T_{\W_{i,n}}(\mu)$:

\begin{prop}\label{L: length estimate}
  With the above notations, there exists a sequence $a_n\to 0$ so that the
  speed of the diagrams $T_{\W_{i,n}}(\mu)$ with respect to the generating
  set $S_n$ satisfies
  \[
  \ell_{S_n}(T_{\W_i,n}(\mu))\leq a_n,
  \]
  for every $i=1,\ldots r(n)$.
\end{prop}

Let us explain how this concludes the proof of \cref{T:bounded Liouville}.
Since the generating sets $S_n$ have bounded cardinalities, we deduce from
\eqref{E: fundamental inequality} that $h(T_{\W_{i,n}}(\mu))\leq C a_n$ for
a constant $C$ that does not depend on $n$.  Since the cardinalities of
$\W_{i,n}$ and $r(n)$ are uniformly bounded, \eqref{main inequality}
implies that there exists $C'>0$ so that $h(\mu)\leq C' a_n \to 0$.  Hence
$h(\mu)=0$.  As noted, this is equivalent to the Liouville property for
$(G,\mu)$ (see \cite{KV83, Dierrenic}).  The rest of this section contains
the proof of this speed estimate.

\begin{proof}[Proof of \cref{L: length estimate}]
  Let $(\al{g}_k, \al{v_k})$ be a random walk with internal degrees of
  freedom with diagram $T_{\W_{i,n}}(\mu)$ starting from $(e,v)$, where
  $v\in\W_{i,n}$ is arbitrary.  Let $|\cdot|$ be the word metric on
  $G^{(n)}$ with respect to the generating set $S_n=\A_n\cup\B_n$. We shall
  prove that there exists a sequence $a_n\to 0$ such that for every $k\geq
  0$ we have
  \[
  \E |\al{g}_k| \leq 1 + a_n k,
  \]
  uniformly in $i$ and the starting point $v\in\W_{i,n}$.  To
  simplify the notations we will henceforth omit the index $i=1,\ldots,
  r(n)$, writing $\W_n, \Abb_n, \Bbb_n$ for $\W_{i,n},
  \Abb_{i,n},\Bbb_{i,n}$.

  Let $\al{h}_1, \dots, \al{h}_k$ be the increments $\al{h}_j =
  \al{g}_{j-1}^{-1}\al{g}_{j}$.  Recall that, conditionally to the
  positions of $\al{v}_j, \al{v}_{j+1}$ the distribution of the increment
  $\al{h}_{j+1}$ is given by the edge measure $\mu_{\al{v}_j,\al{v}_{j+1}}$
  of the diagram $T_{\W_n}(\mu)$, which satisfies \cref{C:mu_properties}.
  We consider two types of ``bad events'' that may happen at some times
  $j\geq 1$.
  \begin{description}
  \item[Traverse:] One of $\al{v}_{j-1}, \al{v}_{j}$ belongs to
    $\Abb_n$ and the other to $\Bbb_n$.
  \item[Bad alignment:] Both $\al{v}_{j-1}, \al{v}_{j}$ belong to
    $\Abb_n$ and $\gamma_{\al{v}_{j-1}}\cdot\al{h}_{j} \neq
    \gamma_{\al{v}_{j}}$.  (Recall the notation $\gamma_v$ from \cref{R:
      gamma v}.)
  \end{description}

  Let $N_k$ be the total number of bad events of either type up to time
  $k$.  We divide the proof of \cref{L: length estimate} into three steps,
  given by \cref{step1,L: bad event 1,L: bad event 2} below, stating that
  the word length is bounded by the number of bad events, and that the
  propabiblity of the two types of bad events at each step is small.

  Note that the probability of of a bad event happening at the $j$th step
  conditionally to $\al{v}_j=v$ only depends on $v$ and on the diagram
  $T_{\W_n}(\mu)$.  We prove that there exists a sequence $a_n\to 0$ so
  that this conditional probability is bounded above by $a_n$, uniformly in
  $v\in \W_{n}$.  This implies $\E N_k\leq c_nk$, concluding the proof of
  \cref{L: length estimate} by \cref{step1}.
For both types of bad events, the proof of this bound is based on \cref{P: recurrence}, which
  applies since $G$ acts on $\partial \Tmb$ with recurrent Schreier graphs
  (see \cref{T:bounded recurrent}).

  \begin{lemma}\label{step1}
    The word metric $|\al{g}_k|$ is bounded above by $1 + N_k$.
  \end{lemma}

  \begin{proof}
    Let $s<t$ be such that  no bad event happens for $j\in [s,t]$. Then either $v_j\in\Abb_n$
    for all $j\in[s,t]$, or else $v_j\in\Bbb_n$ for all $j\in[s,t]$.  In the second
    case, the increment $\al{h}_j$ belongs to the finite group $\B_n$, hence
    their product has length 1 with respect to the generating set
    $S_n=\A_n\cup\B_n$. In the first case, since there is no bad alignment, we have
    $\gamma_{\al{v}_{j-1}}\cdot\al{h}_j=\gamma_{\al{v}_{j}}$ for every
    $j\in[s,t]$. Assume that the second case holds. By \cref{C:mu_properties}, for every
    $j$ the couple $(\gamma_{\al{v}_{j-1}}, \al{h}_{j})$ belongs to the
    finite groupoid $\Ag_n$.  The condition that
    $\gamma_{\al{v}_{j-1}}\cdot\al{h}_{j}=\gamma_{\al{v}_{j}}$ guarantees
    that the product of two consecutive such couples is defined in the
    groupoid, hence belongs to $\Ag_n$.  It follows that
    $(\gamma_{\al{v}_{s-1}}, \al{h}_{s})\cdots
    (\gamma_{\al{v}_{t-1}}, \al{h}_{j_2})=(\gamma_{\al{v}_{t}},
    \al{h}_{s}\cdots\al{h}_{t})\in\Ag_n$ and thus
    $\al{h}_{j_1+1}\cdots\al{h}_{j_2}\in\A_n$ has length 1.

    We conclude that the word length of $\al{g}_k=\al{h}_1\cdots\al{h}_k$
    is bounded by one more than the total number of bad events.
  \end{proof}

  Recall that $(\mu_{vw})$ and $(p_{vw})$ denote the edge measures and the
  marginal transition probabilities of the ascension diagram
  $T_{W_n}(\mu)=(M,P)$.

  \begin{lemma}\label{L: bad event 1}
    There exists a sequence $a'_n \to 0$ so that
    \[
    P(\Bbb_n, \Abb_n) = P(\Abb_n, \Bbb_n) := \sum_{v\in \Abb_n, w\in\Bbb_n}
    p_{vw} \leq a'_n.
    \]
  In particular the probability that a traverse happens at any time is
  bounded above by $a'_n$.
  \end{lemma}

  \begin{proof}
    First, observe that the matrix $P$ is symmetric, as it is the trace of
    a symmetric Markov chain on a recurrent subset. Hence $P(\Bbb_n,
    \Abb_n) = P(\Abb_n, \Bbb_n)$.

    We now argue this quantity is small.  Fix some $\gamma \in
    \Abb_\infty$, and let $v\in \Abb_n$ be the unique vertex that belongs
    to $\gamma$.  We shall prove that $P(v,\Bbb_n):= \sum_{w\in\Bbb_n}
    p_{vw}$ tends to zero as $n\to\infty$.  This is sufficient, since there
    are finitely many choices for $\gamma$.

    Let $(\tilde{\al{g}}_j)$ the random walk on the group $G$ with step
    measure $\mu$.  Recall from the definition of the ascension diagram (see
    \cref{RWIDF and groups acting on trees}) that $P(v, \Bbb_n)$ is the
    probability that $v\cdot \tilde{\al{g}}_\al{t}\in \Bbb_n$, where
    $\al{t}$ is the first return time of $v\cdot \tilde{\al{g}}_j$ to
    $\W_n$.

    Consider simultaneously level $n-K$, and recall that (from the
    definition of $K$) the projection of $\W_n$ to this level is contained
    in $\Abb_{n-K}$.  Let $w\in\Abb_{n-K}$ be the projection of $v$.  Let
    $\al{T}_{n-K}$ be the first return time of $w\cdot\tilde{\al{g}}_j$ to
    $\Abb_{n-K}$ and $\al{T}_n$ be the first return time of
    $v\cdot\tilde{\al{g}}_j$ to $\Abb_n$.  We now apply \cref{P:
      recurrence} to the family of rays $\Abb_\infty$ with starting point
    $\gamma$.  Since $w\cdot\tilde{\al{g}}_\al{t}$ is the projection of
    $v\cdot\tilde{\al{g}}_\al{t}$ and the latter is in $\W_n$ it follows
    that $w\cdot\tilde{\al{g}}_\al{t}\in\Abb_{n-K}$, and in particular
    $\al{T}_{n-k} \leq \al{t}$.

    Suppose $v\cdot\tilde{\al{g}}_\al{t}\in \Bbb_n$, then
    $(v\cdot\tilde{\al{g}}_j)$ returns to $\Abb_n$ strictly after time
    $\al{t}$.  Hence $\al{T}_{n-K}\leq \al{t}< \al{T}_n$. By \cref{P:
      recurrence} the probability of this event tends to $0$ as
    $n\to\infty$, concluding the proof.
  \end{proof}

  \begin{lemma}\label{L: bad event 2}
    There exists a sequence $a''_n \to 0$, so that for every
    $v\in\Abb_n$ we have $q_v\leq a''_n$, where
    \[
    q_v := \sum_{w\in\Abb_n} p_{vw} \mu_{vw} \{h : \gamma_v\cdot h\neq
    \gamma_w\}
    \]
    is the probability that a bad alignement event happens at the $j$th step
    conditioned on $\al{v}_{j-1}=v$.
  \end{lemma}

  \begin{proof}
    This proof too is based on \cref{P: recurrence}.  Fix again $\gamma\in
    \Abb_\infty$ and let $v\in\Abb_n$ be the unique vertex belonging to
    $\gamma$.  We assume that $n$ is large enough so that \cref{R: gamma v}
    applies.  As before, it is enough to prove that $q_v$ tends to zero as
    $n\to \infty$ and $v$ belongs to $\gamma$.  Let $(\tilde{\al{g}}_j)$ be
    the random walk on $G$ with step measure $\mu$ and let $\al{t}$ be the
    first return time of $v\cdot\tilde{\al{g}}_j$ to $\W_n$. Let $\al{T}_n$
    (resp.\ $\al{T}_\infty$) be the return time of $v\cdot\tilde{\al{g}}_j$
    (resp.\ $\gamma\cdot\tilde{\al{g}}_j$) to $\Abb_n$ (resp.\
    $\Abb_\infty$).  Setting $\al{w}:=v\cdot\tilde{\al{g}}_\al{t}$ the
    probability $q_v$ equals the probability that $\al{w}\in\Abb_n$ and
    $\gamma_v\cdot \tilde{\al{g}}_\al{t}|_v\neq \gamma_\al{w}$.  If this
    event happens we have $\al{T}_n=\al{t}$, while $\al{T}_\infty>\al{t}$.
    Indeed, $\gamma\cdot \tilde{\al{g}}_\al{t}= (\gamma_v\cdot
    \tilde{\al{g}}_\al{t}|_v)\al{w}\neq \gamma_\al{w}\al{w}$ and hence
    $\gamma\cdot \tilde{\al{g}}_\al{t}\notin\Abb_\infty$ (since
    $\gamma\cdot \tilde{\al{g}}_\al{t}$ contains $\al{w}$, but the unique
    ray in $\Abb_\infty$ that contains $\al{w}$ is $\gamma_\al{w}\al{w}$).
    This implies that $\al{T}_n<\al{T}_\infty$.  The probability of this
    event tends to zero as $n\to\infty$ by \cref{P: recurrence}.
  \end{proof}

  Setting $a_n=a'_n+a''_n$ we have $\E N_k\leq a_n k$.  This concludes the
  proof of \cref{L: length estimate}, and thus the proof of \cref{T:bounded
    Liouville} as noted.
\end{proof}

\section{Proof of \cref{T:spherically homogeneous}}
\label{sec:proof thm 2}

To make the proof of \cref{T:bounded Liouville} quantitative, the key idea
is to let the level $n$ tend to infinity together with the time $k$, at a
carefully chosen rate. One needs an estimate on the rate of convergence to
0 of the probabilities from \cref{L: bad event 1,L: bad event 2}.  Such
estimates can be obtained from a closer analysis of the Schreier graphs of
the action of $G$ on the finite levels of the tree using electric network
theory.

This section is organized as follows: In \cref{S:directed} we define
principal groups of directed automorphisms, and in
\cref{S:Sections_principal} we study sections in such groups.  Some
simplifications occur in this setting,  in particular the groupoid $\mathcal{A}_n$ can be chosen to be a group, and the bad alignment events cannot occur.  Then in
\cref{S:Resistances} we calculate lower bounds on the resistance in the
relevant Schreier graphs, and finally in \cref{S:entropy estimate} we
combine all ingredients to prove \cref{T:spherically homogeneous}.

\subsection{Principal groups of directed automorphisms and the mother group}
\label{S:directed}
\label{mother group}

Let $\mb$ be a bounded sequence as before, and set $m_*=\max_i m_i$.  The
$0$-ray in $\Tmb$, consists of all vertices of the form $0^n=0\dots0$.  The
neighbours of the zero ray in the tree are vertices of the form $v=x0^n$,
where $x\in X_{m_i}$ is the only non-zero letter in $v$.  Let

$H_\mb < \Aut(\Tmb)$ be the subgroup consisting of elements that are
directed along the zero ray, fix the zero ray, and have depth at most one
(recall that this means that their sections can be non-trivial only on the
zero ray or its neighbours).  Equivalently, an element $h\in\H_\mb$ has a
wreath recursion of the form
\[
h = (h',\tau_1,\cdots,\tau_{m_1-1})\rho,
\]
where $\tau_1,\dots,\tau_{m_1-1} \in \S_{m_2}$ are the sections at first-level vertices other than  $0$, the permutation $\rho\in\S_{m_1}$ is such
that $0\cdot\rho=0$ and $h'\in H_{\sigma\mb}$. The group $H_\mb$ is
locally finite.

We identify the symmetric group $S_{m_1}$ with the subgroup of $\Aut(\Tmb)$
consisting of automorphisms that permute vertices on the first level and
have trivial sections on them.  In the same way, $S_{m_n}$ identifies with
a subgroup of $\Aut(\T_{\sigma^n\mb})$ for every $n\geq 0$.

\medskip

The following groups were defined and studied by Brieussel (see
\cite{Brieu:nonuniformgrowth,Brieu:entropy,Brieu:folnersets}).  They are a
generalization of the mother group from \cite{BKN10}.

\begin{defin}\label{def:group of directed automorphisms}
  Let $A < \H_{\mb}$, $B < S_{m_1}$ be finite subgroups.  The \emph{the
    principal group of directed automorphisms} generated by $A$ and $B$ is
  the group $\langle A\cup B\rangle < \Aut(\Tmb)$.  We denote it by
  $\M(A,B)$.
\end{defin}
Note that the term "principal group of directed automorphisms" should be taken as whole, in fact the group $M(A, B)$ is generated by directed automorphisms but also contains automorphisms that are not directed.

Many groups acting on $\T_m$ embed in a group of
the form $M(A, B)$, see \cref{embedding mother group} below and also \cite[Section
9]{Brieu:nonuniformgrowth} for a slight generalization.  We shall omit
$A,B$ from the notation when there is no ambiguity, and write simply $\M$
for $\M(A,B)$.

There is an important particular case of \cref{def:group of directed
  automorphisms}.  Take $\mb$ a constant sequence, and set $B=S_m$.  For
$A$ we take all elements $h\in\H_m$ for which the section at $0$ is $h$
itself: $h|_0=h$.  Equivalently, $A$ is the group of automorphisms that
admit a wreath recursion of the form
\[
h = (h,\tau_1,\dots,\tau_{m-1}) \rho,
\]
where $\tau_1,\dots,\tau_{m-1} \in \S_m$ and $\rho\in\S_m$ is such that
$0\cdot\rho=0$.  It is easy to see that $h$ is determined by $\rho$ and the
$\tau_i$, and that $A$ is a finite group, isomorphic to $S_m
\wr_{X\setminus\{0\}} S_{m-1}$.  With these choices of $A$ and $B$, the
group $\M=\M(A,B)$ is generated by a bounded automaton, and is called the
\emph{mother group of bounded activity} over the $m$-element alphabet.

The mother group was first defined in \cite{BKN10} in the bounded activity
case.  An analogous generalization to higher activity degrees was provided
in \cite{AAV:amenability}.  Its significance relies on the fact that every
polynomial activity automaton group embeds in a mother group of the same
activity degree, possibly acting on a bigger alphabet.  We only use this
result in the bounded activity case:

\begin{thm}[\cite{BKN10,AAV:amenability}]\label{embedding mother group}
  Let $G < \Aut(\T_m)$ be a group generated by a bounded activity
  automaton.  Then there exists $m'$ such that $G$ embeds isomorphically in
  the mother group of bounded activity over $m'$ elements.
\end{thm}

Henceforth, we shall fix a sequence $\mb=(m_n)_n$ of natural numbers bounded
by $m_*=\max_n m_n$, as well as two finite groups $A<\H_{\mb}$ and
$B\subset S_{m_1}$ generating a principal group of directed automorphisms
$\M = \M(A,B)$.  We also fix a subgroup $G < \M$, generated by a finite
symmetric set $S\subset \M$.

Furthermore, it will be useful to suppose that $A$ contains the following
elements.  Let $\bar{\sigma}=(\sigma_1,\dots,\sigma_{m_*})$ with
$\sigma_i\in S_i$ be a collection of permutations in the symmetric groups
up to $m_*$ elements.  Define $h_{\bar{\sigma}}\in\Aut(\Tmb)$ to act on
words as follows.  If the first non-zero letter of word $w$ is at position
$i$, then $w\cdot h_{\bar{\sigma}}$ is equal to $w$ except for the $i+1$st
letter which is permuted by $\sigma_{m_{i+1}}$.  It is easy to see that
elements of the form $h_{\bar{\sigma}}$ are in $\H_{\mb}$ and form a finite
group.  We shall suppose that $\A$ contains this finite group.  Adding any
finite set of elements to $\A$ does not cause any loss of generality, since
the group $\H_{\mb}$ is locally finite.

\subsection{Sections in the principal groups of directed automorphisms}
\label{S:Sections_principal}

We now describe the sections of the generators $s\in S$ of the group $G$.
We will use notations analogous to those in \cref{S:sections_bounded}.

\begin{defin}\label{sections of directed group}
  \begin{itemize}
  \item Let $A_n = \langle a|_{0\cdots0}\rangle_{a\in A}$ the finite subgroup
    of $\H_{\sigma^n\mb}$ consisting of sections of elements of $A$ at
    $n$th level along the zero ray.
  \item Let $B_{n}=\langle a|_{x0\cdots0}\rangle_{a\in A,\: x\in
      X_{m_n}\setminus \{0\}}=\langle a'|_x\rangle_{a'\in A_{n-1},x\in
      X_{m_n}\setminus\{0\}} $ the subgroup of $\S_{m_n}$ generated by the
    $n$th level sections of $A$ at neighbors of the zero ray.
  \end{itemize}
\end{defin}

Note that $A_n\cup B_n$ generate the group of $n$ level sections
$\M^{(n)}$.  In particular, $\M^{(n)}=M(A_n, B_n)$ is a principal group of directed
automorphisms of $\T_{\sigma^n\mb}$.

As in \cref{S:sections_bounded} we denote by $\W_n\subset
\Tmb^n$ the set of $n$th level vertices $w\in \Tmb^n$ such that the section
$s|_w$ is non-trivial for some generator $s\in S$.  We also keep the same
definitions of the set of singular rays $\Abb_\infty$ and the sets
$\Abb_n,\Bbb_n$. The following Lemma is a more explicit version of
\cref{P:sections} in this setting.

\begin{lemma}\label{L:sections of generating set}
  The set $\Abb_\infty$ consists of rays ending with an infinite sequence
  of zeros.  In particular there is an $n_0$ and set $\{w_1,\dots,w_k\}
  \subset \Tmb$ independent of $n$, so that for $n>n_0$, the sets $\Abb_n$
  and $\Bbb_n$ have the form
  \begin{itemize}
  \item $\Abb_n=\{00\dots0w_j \}$,
  \item $\Bbb_n=\{x00\dots0w_j : x\in X_{m_n}\setminus\{0\}\}$.
  \end{itemize}
  Moreover for every generator $s$, we have $s|_w\in A_n$ (resp.\ $s|_w\in
  B_n$) if $w\in \Abb_n$ (resp.\ $w\in\Bbb_n$) and $s|_w=e$ otherwise.
\end{lemma}

\begin{proof}
  We first show that for every $h\in\M$ there exist a $n_h$ such that for
  $n\geq n_h$ all of its $n$th level sections are in the generating set
  $A_n\cup B_n$.  From the definition of $A_n$ and $B_n$, it suffices to
  prove this for $n=n_h$.  We do this by induction on the word metric $|h|$
  associated to the generating set $A\cup B$. For $h\in A\cup B$ the claim
  holds with $n_h=0$.

  First of all, observe that if $h$ is a product of two generators then its
  first level sections are in $A_1\cup B_1$, so that one can take $n_h=1$.
  Indeed, if $h=s_1s_2$, with $s_1,s_2\in\A$ then $h\in\A$ and its first
  level sections are in $A_1\cup B_1$ by definition.  If $s_1$ is in $B$
  then its sections are trivial, and from \eqref{sections multiplication
    rule} we see that first level sections of $h$ are those of $s_2$
  possibly in a different order, and are in $A_1 \cup B_1$.  Similarly,
  this is the case if $s_2\in B$.

  Generally, suppose that the conclusion holds for $h$, and consider
  $g=hs$.  Then sections of $g$ at level $n_h+1$ are first level sections
  of products from $A_{n_h}\cup\B_{n_h}$.  The case of a product of two
  generators applies, and the sections are in $\A_{n_h+1}\cup\B_{n_h+1}$,
  so $n_g=n_h+1$ will do.

  \medskip

  We deduce that for every large enough level $n$, the sections of every
  generator $s\in S$ are in $A_n\cup B_n$. To conclude observe that
  elements of $B_n$ are finitary and elements of $A_n$ are directed along
  the zero ray. It follows that the singular set $\Abb_\infty$ of generators
  consists of rays ending with an infinite sequence of zeros, and that
  $\Abb_n$ and $\Bbb_n$ have the claimed form.
\end{proof}

\subsection{Resistances in Schreier Graphs}
\label{S:Resistances}

\renewcommand{\G}{\Lambda}

In this subsection we analyze effective resistances in the Schreier graphs
of the group $G$ acting on the levels of the tree, with respect to the
fixed generating set $S$. See \cite[Chapter 2]{LyPe} for a general
background on electric network theory.

It is convenient to first consider the Schreier graph $\G_n$ for the whole
group $\M$ acting on the $n$th level $\Tmb^n$, equipped with the standard
generating set $A\cup B$.  Call the vertex $0^n\in\Tmb^n$ the \emph{root}.
Vertices of the form $x0^{n-1}\in\Tmb^n$ with $x\neq0$ are called the
\emph{anti-roots}.
The following proposition determines a lower bound for the asymptotics of
resistance in $\G_n$ between the root and any anti-root, as $n\to\infty$.
See also \cite{AV:positivespeed,AV:speedexponents} for more on resistances
in these graphs.

\begin{lemma}\label{L:resistance}
  There exists a constant $c$, not depending on $n$, such that for every
  $x\neq 0$ we have the resistance bound
  \[
  \res_{\G_n} (0^n, x0^{n-1})
  \geq c \prod_{i=1}^n \frac{m_i}{m_i-1}
  \geq c \left(\frac{m_*}{m_*-1}\right)^n.
  \]
\end{lemma}

\begin{proof}
  A word in $\Tmb^n$ can be mapped to a word in $\{0,*\}^n$, by
  substituting every non-zero letter with the symbol $*$.  The set of
  anti-roots is exactly the pre-image of $*0^{n-1}$.  The graph $\G_n$
  projects to a graph $\hat{\G}_n$ with vertex set $\{0,*\}^n$ and multiple
  edges.  By Rayleigh monotonicity, resistances in the projected graph are
  no larger than resistances in the original graph.

  The key observation is that $\hat{\G}_n$ is just a path with multiple
  edges, and some self loops.  The root $0^n$ and anti-root $*0^{n-1}$ are
  the ends of the path.  To see this, observe by looking to the action of
  the generators that there are only two kind of non trivial moves on
  elements of $\{0,*\}^n$: changing the rightmost (first) letter (generators in $B$) or changing the letter after the first appearance of
  $*$ from the right (generators in $A$).  Moves of the second kind give
  loops for the root $0^n$ and the anti-root $*0^{n-1}$, so these vertices
  have only one other neighbour. A connected multi-graph having all vertices
  of degree two, except two vertices of degree 1, is a path with multiple
  edges and loops.

  To get a bound on the resistances in $\G_n$ we need to find the edge
  multiplicities in $\hat{\G}_n$.  The degree of vertices in $\G_n$ is
  bounded by some $C$, so the degree of a vertex $x\in\{0,*\}^n$ is at most
  $C$ times the number of vertices in $\Tmb^n$ that project to $x$, i.e.\
  it is bounded above by $C \prod_{i| x_i=*}(m_i-1)$.  Hence the total
  resistance is bounded below by
  \[
  \res_{\hat{\G}_n}(0^n,*0^{n-1})
  \geq \sum_{\{0,*\}^n} C^{-1} \prod_{i | x_i=*} \frac{1}{m_i-1}
  = C^{-1} \prod_{i=1}^n \frac{m_i}{m_i-1}   \qedhere
  \]
\end{proof}

The Rayleigh monotonicity principle and rough invariance of resistances
under quasi-isometries \cite[Chapter 2]{LyPe} allow us to deduce a similar
consequence for the group $G$ equipped with any symmetric generating set
$S$.  Fix a level $\Tmb^n$ deep enough, and recall the definition of the
vertex sets $\Abb_n$ and $\Bbb_n$ from \cref{L:sections of generating set}.

Recall that the resistance between two vertex sets $\Abb,\Bbb$ in a graph is defined as the
resistance from $\bar{a}$ to $\bar{b}$ in the graph where $\Abb$ and $\Bbb$ have
been collapsed to points $\bar{a},\bar{b}$ (this definition makes sense
also for disconnected graphs).

\begin{lemma}\label{L:resistance orbit}
  Consider a fixed generating set $S$ of $G$, and let $\Gamma_n$ be the
  (possibly disconnected) corresponding Schreier graph of the action of $G$
  on $\Tmb^n$.  There exists a constant $c$ depending only on the
  generating set $S$ such that for any large enough $n$
  \[
  \res_{\Gamma_n}(\Abb_n,\Bbb_n) \geq c \prod_{i=1}^n \frac{m_i}{m_i-1}\geq c \left(\frac{m_*}{m_*-1} \right)^n.
  \]
  Moreover, the same holds if each $s\in S$ has some conductance which applies
  to the corresponding edges in $\Gamma_n$.
\end{lemma}

\begin{proof}
  Since $\Abb_n, \Bbb_n$ have uniformly bounded cardinalities, it is
  sufficient to prove that for any $v\in \Abb_n$ and $w\in \Bbb_n$, the effective resistance between $v$ and $w$ in
  $\Gamma_n$ satisfies the same bound (with a possibly larger constant).
  Consider the larger generating set $\tilde{S} = S\cup A\cup B$ of $\M$,
  and let $\tilde\Gamma_n$ be the corresponding Schreier graph of the
  action of $\M$ on $\Tmb^n$.  Then $\Gamma_n$ and $\G_n$ are both subgraph
  of $\tilde\Gamma_n$, and by Rayleigh monotonicity
  \[
  \res_{\Gamma_n}(v,w) \geq \res_{\tilde\Gamma_n}(v,w).
  \]
  Next, note that the graph $\tilde{\Gamma}_n$ is roughly equivalent
  \cite[p.51]{LyPe} to the standard Schreier graph $\G_n$ with constants
  not depending on $n$.  Thus effective resistances of the graphs
  $\tilde\Gamma_n$ and $\G_n$ are equivalent up to multiplicative
  constants.  With \cref{L:resistance}, this implies that for every $x\neq
  0$
  \[
  \res_{\tilde\Gamma_n} (0^n,x0^{n-1}) \geq c'\prod_{i=1}^n \frac{m_i}{m_i-1}
  \]
  for some constant $c'$.

  Finally, we shall show that there is some constant $K$, so that for
  $v\in\Abb_n$ and $w\in\Bbb_n$, the distance in $\tilde\Gamma_n$ from
  $0^n$ to $v$ (and similarly from $x0^{n-1}$ to $w$) are at most $K$.
  Since resistance is bounded by distance and by the triangle inequality
  for resistances we get
  \[
  \res_{\Gamma_n}(v,w) \geq c'\prod_{i=1}^n \frac{m_i}{m_i-1} - 2K.
  \]
  Taking $n$ large enough, this completes the proof.  Since $\G_n$ is a
  subgraph of $\tilde\Gamma_n$ it suffices to bound distances in $\G_n$.

  By \cref{L:sections of generating set} we have $v=0\dots0w_i$ and
  $w=x0\dots0w_j$ where $w_i$ and $w_j$ are words in some fixed and finite
  set. Not that there is an element $g$
  of length at most $2^l-1$  such that $0^l\cdot g =
  w_i$, and $g|_{0^l}=e$, this is easy to see by induction on $l$ using only the generators of $B$ and
  $h_{\bar\sigma}$ (which we assumed to be in $A$).  It follows that the distance from $0^n$ to $v$
  is at most $2^l-1$ for any $n$, and similarly for $w$ and $x0^{n-1}$.

  If the elements of $S$ have associated conductances, then these are
  bounded by some constant.  By monotonicity, the resistance is decreased
  by at most that constant.
\end{proof}

\subsection{Entropy estimate}\label{S:entropy estimate}

We keep notations from the previous section: $\mu$ is supported on the set
$S$, and we denote by $\Gamma_n$ the Schreier graph of the action of
$G=\langle S\rangle$ on the level set $\Tmb^n$ with generating set $S$.  If
the action is not transitive $\Gamma_n$ is not connected, but this is
irrelevant in what follows.  Let $V_n=m_1\cdots m_n$ be the volume of the
level sets of the tree.  With $\Abb_n$ and $\Bbb_n$ as above
(\cref{L:sections of generating set}), we shall consider two resistances:
\begin{align*}
  R_n &= \res_{\Gamma_n}(\Abb_n,\Bbb_n),  &
  R^\mu_n &= \res_{(\Gamma_n,\mu)}(\Abb_n,\Bbb_n).
\end{align*}
Here $R_n$ is computed in the graph with all edge weights equal to $1$, and
$R^\mu_n$ is the resistance with edge weights given by $\mu$.  Since
$\mu(g)\le 1$ for any $g$ we have $R_n^\mu \geq R_n$.

We will use the following slight generalization of the classical formula
for random walk hitting probabilities (in the case when $A=\{a\}$ is a
singleton).  We have not located a reference for this, but this is a
reformulation of Exercise 2.45 in \cite{LyPe}.

\begin{lemma}\label{L:commute}
  Let $\Gamma$ be a finite weighted graph (possibly disconnected) where each edge $e$ has weight $w_e$.  Set $Q
  = \sum_e w_e$, the total weight of the graph.  Consider the random walk
  $(X_n)_{n \geq 0}$ started at its stationary measure $\nu(x)=Q^{-1} \sum_{e\ni
    x} w_e$.  For a vertex set $W\subset \Gamma$ denote the hitting time by
  $T_W = \min\{n\geq 1 : X_n\in W\}$.  Then for disjoint vertex sets $A,B$
  \[
  \P(X_0\in A, T_B<T_A) = \frac{1}{2Q\res(A,B)}.
  \]
\end{lemma}

The following proposition is a quantitative version of our previous proof.

\begin{prop}\label{L:explicit estimate}
  With the above notations, there exists a constant $C$ depending only on
  $\mb$, on the ambient group $\M(\A,\B)$ and on $\supp(\mu)$ such that for
  every $n$ and $k$ we have
  \[
  H(\mu^{*k}) \leq C \left(V_n + \frac{k}{R_n} \right).
  \]
\end{prop}

Before proving \cref{L:explicit estimate}, let us see how this implies
\cref{T:spherically homogeneous} by an appropriate choice of the level $n$.
Let $n=n(k)$ be the smallest integer such that $k\leq V_n
R_n$. \cref{L:explicit estimate} applied for this choice of $n$ gives
\[
H(\mu^{*k})\leq 2 C V_n= 2C\prod_{i=1}^n m_i.
\]
Recall that $\alpha=\log m_*/ \log\frac{m_*^2}{m_*-1}$ is as in the theorem
and note that for $m\leq m_*$ we have $m\leq \left( \frac{m^2}{m-1}
\right)^\alpha$.  Recall also from \cref{L:resistance orbit} that $R_n \geq
c\prod_{i=1}^n \left(\frac{m_i}{m_i-1}\right)$.  Using these inequalities, we get
\begin{align*}
  H(\mu^{*k})
  &\leq 2C m_* \prod_{i=1}^{n-1} m_i
  \leq 2C m_* \left(\prod_{i=1}^{n-1}\frac{m_i^2}{m_i-1}\right)^\alpha \\
  &= 2C m_* \left(V_{n-1}\prod_{i=1}^{n-1}\frac{m_i}{m_i-1}\right)^\alpha
  \leq 2C m_* (V_{n-1}R_{n-1})^\alpha \leq 2C m_* k^\alpha,
\end{align*}
where we used that $V_{n-1}R_{n-1}\leq k$ by the choice of $n$.

\begin{proof}[Proof of \cref{L:explicit estimate}]
  Fix $k$ and $n$ big enough, and let $\al{g}_k$ be the $k$th step of a
  random walk on $G$ with law $\mu^{*k}$. We have that $\al{g}_k$ is
  determined by its action $\sigma_k$ on $\Tmb^n$ together with its
  sections at the vertices there, and we shall estimate the entropy coming
  from each of these parts.

  Consider the finite tree $\tilde{\T}_\mb^n$ consisting of all vertices up
  to and including level $n$, and let $\tilde V_n \leq 2V_n$ be its
  cardinality. An automorphism of $\tilde{\T}_\mb^n$ is determined by the
  permutation associated to each of its vertices, that give the action on
  the children (this is called the {\em portrait} of the automorphism).
  Hence the set of automorphisms of the finite tree has cardinality at most
  $(m_*!)^{\tilde{V}_n} \leq C^{V_n}$, where $C$ depends only on $m$. It
  follows that $\sigma_k$ has at most $C^{V_n}$ possible values, and by
  \cref{P:entropy}(1) we have
  \[
  H(\sigma_k) \leq C V_n
  \]
  for some $C$ depending only on $m$.

  To make the next estimates cleaner it is convenient to add randomness in
  the form of an independent uniform automorphism $\epsilon$ of
  $\tilde{\T}_\mb^n$.  As for $\sigma_k$, we have $H(\epsilon) \leq C V_n$. We then have
  \begin{align*}
    H(\al{g}_k) \leq H(\al{g}_k,\epsilon)
    = H(\sigma_k, \epsilon, (\al{g}_k|_v)_{v\in\Tmb^n}) &= H(\sigma_k, \epsilon, (\al{g}_k|_{\epsilon (v)})_{v\in\Tmb^n}) \\
    &\leq H(\sigma_k) + H(\epsilon) + \sum_{v\in\Tmb^n}
    H(\al{g}_k|_{\epsilon(v)}).
  \end{align*}
  The first two terms here are at most $CV_n$.  The advantage of using
  $\epsilon$ independent of $\al{g}_k$ is that $\epsilon(v)$ is uniform in
  $\Tmb^n$, and in particular all terms in the last sum are now equal and
  the sum equals $V_nH(\al{g}_k|_\al{v})$, where $\al{v}$ is a uniform
  random vertex of $\Tmb^n$.  (Alternatively, this could be achieved with
  an $\epsilon$ with smaller entropy $\log V_n$, by taking a random power
  of some fixed cyclic permutation of $\Tmb^n$.)

  To estimate $H(\bg_k|_\bv)$ we note that $(\bg_i|_\bv,\bv\cdot\bg_i)$ is
  a random walk with internal degrees of freedom with state space $\Tmb^n$,
  and that the edge measures are supported on the finite groups $\A_n$ and
  $\B_n$ when $\bv\cdot\bg_i$ is at $\Abb_n$ and $\Bbb_n$ respectively, and
  on the identity otherwise (\cref{C:mu_properties}). It follows from
  \cref{P:entropy}(3) that $H(\al{g}_k|_\al{v}) \leq C \E
  |\al{g}_k|_\al{v}|+C$, where the length $|\al{g}_k|_\al{v}|$ is measured
  w.r.t.\ the generating set $\A_n \cup \B_n$, and the constant $C$ depends
  only on the cardinalities of these groups (hence on $\M(\A,\B)$ only).

  Let $\ell$ be the number of times the walk $(\bv\cdot\bg_i)$ moves from
  $\Abb_n$ to $\Bbb_n$ and back to $\Abb_n$ up to time $k$. Then we have
  $|\al{g}_k|_\al{v}| \leq 2\ell+2$, and so we need to estimate $\E\ell$.
  Say that a traverse begins at time $i$ if $\bv\cdot\bg_i\in\Abb_n$ and if
  the walk then visits $\Bbb_n$ before returning to $\Abb_n$. Note that we
  do not care whether the visit to $\Bbb_n$ or the return to $\Abb_n$ occur
  before or after time $k$. We now use \cref{L:commute}, which applies
  since $\bv\cdot\bg_i$ is stationary. The total weight of edges leaving
  each vertex is $1$, and so (recalling that $R^\mu_n\geq R_n$)
  \[
  \P(\text{a traverse begins at time } i) = \frac{1}{2V_n R_n^\mu}\leq
  \frac{1}{2V_n R_n}.
  \]
  Thus $\E\ell \leq \frac{k}{2V_n R_n}$, and so
  \[
  H(\al{g}_k|_\al{v}) \leq C\E(2\ell+2) \leq \frac{Ck}{V_n R_n} + 2C,
  \]
  and
  \[
  H(\bg_k) \leq C'\left(V_n + \frac{k}{R_n}\right).  \qedhere
  \]
\end{proof}

\let\oldthebibliography=\thebibliography
\let\endoldthebibliography=\endthebibliography
\renewenvironment{thebibliography}[1]{
  \begin{oldthebibliography}{#1}
    \setlength{\itemsep}{.5mm}
  }{
  \end{oldthebibliography}
}

\bibliography{rootedtrees}
\bibliographystyle{alpha}

\end{document}